\documentclass[12pt,a4paper]{amsart}

\usepackage[utf8]{inputenc}
\usepackage[english]{babel}

\usepackage{latexsym}
\usepackage{amsfonts}
\usepackage{amsmath}
\usepackage{amssymb}
\usepackage{graphicx}
\usepackage{color}
\usepackage{enumitem}

\usepackage{amsthm}
\usepackage{pstricks}
\usepackage{csquotes}
\usepackage{fullpage}
\usepackage{url}
\usepackage[ocgcolorlinks, linkcolor=red]{hyperref}

\usepackage{algorithm}
\usepackage{algcompatible}

\newcommand{\R}{\mathbb{R}}
\newcommand{\N}{\mathbb{N}}

\newcommand{\C}{\mathbb{C}}

\newcommand{\ov}[1]{\overline{#1}}

\newcommand{\Hscurl}[2]{H^{#1}(\operatorname{curl};#2)}
\newcommand{\Hcurl}[1]{H(\operatorname{curl};#1)}

\newcommand\supp{\operatorname{supp}} 
 
\newcommand\spec{\operatorname{Spec}}

\newcommand\p{\partial}

\newcommand\spn{\operatorname{span}} 
 
\newcommand\osupp{\operatorname{osupp}}

\newtheorem{thm}{Theorem}[section]

\newtheorem{lem}[thm]{Lemma}
\newtheorem{prop}[thm]{Proposition}

\newtheorem{rem}[thm]{Remark}

\theoremstyle{definition}

\numberwithin{equation}{section}

\begin{document}

\title[]{The linearized monotonicity method for elastic waves and the separation of material parameters}

\date{}

\author[]{Sarah Eberle-Blick}
\address{Sarah Eberle-Blick, Institute of Mathematics, Goethe University Frankfurt, Germany}
\email{eberle@math.uni-frankfurt.de}

\author[]{Valter Pohjola}
\address{Valter Pohjola, Unit of Applied and Computational Mathematics, University of Oulu, Finland}
\email{valter.pohjola@gmail.com}

\begin{abstract}
We derive a linearized version of the monotonicity method for 
shape reconstruction using time harmonic elastic waves. The linearized
method provides an efficient version of the method, drastically reducing computation time.
Here we show that the linearized method has some additional advantages.
The linearized method can in particular be used to 
obtain additional information on the material parameters, 
and is able to partially separate and identify the supports of the Lamé parameters.
\end{abstract}

\maketitle

\tableofcontents

\section{Introduction} 
\noindent
Waves in an elastic body provide a natural way to obtain information about its structure,
which leads to various forms of inverse problems.
One such problem is to identify the shape of a penetrable inclusion from measurements at the boundary. 
This is of great interest in various imaging applications, and the
mathematical treatment of this problem has received considerable attention in various physical settings. 
See e.g. \cite{CK98,GK08,CC06, Pot06} and the references therein. 
Here we investigate the monotonicity method for this inclusion reconstruction problem in an elastic body
using time harmonic waves continuing our work in \cite{EP24}.
\\
\\
The elastic properties of an isotropic elastic body $\Omega \subset \R^3$ is in the linear regime
characterized  by the Lamé parameters $\lambda$ and $\mu$.
The effect of a time harmonic oscillation in such a body $\Omega$ is given by
the Navier equation. This describes the oscillations of the displacement field 
$u : \Omega \to \R^3$, $u \in H^1(\Omega)^3$ of the solid body $\Omega$, due to disturbances.
% due to an oscillation and the Lamé parameters.
Here we consider a material body $\Omega \subset \R^3$ with a
$C^{1,1}$-boundary\footnote{ The space $C^{1,1}$ consists of Lipschitz
continuous functions with a Lipschitz continuous derivative.}, % See e.g. Adams book
and consider the Navier equation in terms of the following boundary 
value problem
\begin{align}  \label{eq_bvp1}
\begin{cases}
\nabla \cdot (\C\,  \hat \nabla u )  + \omega^2\rho u &= 0, \quad\text{in}\,\,\Omega,\\
\hspace{1.8cm}(\C\,  \hat \nabla u ) \nu  &= g, \quad\text{on}\,\,\Gamma_N, \\
\hphantom{\hspace{1.8cm}(\C\,  \hat \nabla u ) } u  &= 0, \quad\text{on}\,\,\Gamma_D, %\\
% \text{boundary condition},
\end{cases}
\end{align}
where $\Gamma_N$ and $\Gamma_D$ are such that
$$
\Gamma_N, \Gamma_D \subset \p\Omega \text{ are open }, \qquad \Gamma_N \neq \emptyset, \qquad
\p \Omega = \ov{\Gamma}_N \cup \ov{\Gamma}_D,
$$
and where $\hat \nabla u  = \frac{1}{2}(\nabla u + (\nabla u)^T)$ is the symmetrization of the Jacobian
or the strain tensor, 
and $\C$ is the 4th order tensor defined by
\begin{equation} \label{eq_Cdef}
\begin{aligned}
(\C A)_{ij} = 2\mu A_{ij} + \lambda \operatorname{tr}(A) \delta_{ij},
\quad \text{ where } A \in \R^{3 \times 3},
\end{aligned}
\end{equation}
and $\delta_{ij}$ is the Kronecker delta. Here
$\lambda,\mu \in L^\infty_+(\Omega)$ are the scalar functions that specify the Lamé parameters, which determine
the elastic  properties of the material, $\rho \in L^\infty_+(\Omega)$ is the density of the material,
and $\omega \neq 0$ the angular frequency of the oscillation, and $\nu$
is the outward pointing unit normal vector to the boundary $\p \Omega$. 
The vector field $g \in L^2(\Gamma_N)^3$ acts as the source
of the oscillation, and since $\C \,\hat \nabla u$ equals by Hooke's law to the Cauchy stress tensor, we see 
that the boundary condition $g$ specifies the traction on the surface $\p \Omega$.

We also make the standing assumption that $\omega \in \R$ is not a resonance frequency. When 
this assumption holds, the problem \eqref{eq_bvp1} admits a unique solution
for a given boundary condition $g\in L^2(\Gamma_N)^3$. We can thus define the Neumann-to-Dirichlet map
$\Lambda:L^2(\Gamma_N)^3 \to L^2(\Gamma_N)^3$, as
\begin{equation} \label{eq_ND_map}
\begin{aligned}
\Lambda: g \mapsto u|_{\Gamma_N}. 
\end{aligned}
\end{equation}
Thus $\Lambda$  maps the traction to the displacement $u|_{\Gamma_N}$ on the boundary.

\medskip
\noindent
In this paper we are concerned with formulating a monotonicity based shape reconstruction method
for elasticity as in \cite{EP24} and \cite{EH21}.  
We are more specifically interested in determining the shape of perturbations of the material parameters
$\lambda, \mu$ and $\rho$, in an otherwise homogeneous  background material characterized by the constant
material parameters $\lambda_0, \mu_0, \rho_0 > 0$. We would  ideally want to reconstruct the sets
$$
\supp(\lambda-\lambda_0), \quad
\supp(\mu-\mu_0),  \quad
\supp(\rho-\rho_0)
$$
from the Neumann-to-Dirichlet map $\Lambda$. It is however 
unknown how to solve the problem in this generality.
For this reason we introduce the notion of the outer support for functions and sets.
The outer support of a function $f$ corresponds to
$$
\osupp(f) = ``\text{complement of the component of }
\big\{ x \in \Omega \,:\, f(x) = 0 \big\} \text{ connected to }\p \Omega".
$$
See definition \ref{eq_def_osupp} and \cite{HU13} for details.
For a set $V\subset \Omega$ we set $\osupp(V) = \osupp(\chi_V)$. 
The set $\osupp(V)$ can be thought of as  a set obtained from $V \subset \Omega$ by filling 
in any internal cavities. 

Our main result is the formulation of a linearized and improved version of the monotonicity based shape reconstruction
procedure in \cite{EP24}. This is based on Theorem \ref{thm_Lin_inclusionDetection} and elaborated in sections
\ref{sec_reconstruct_D} and \ref{sec_mu_recons}.
Firstly we show that we can  reconstruct the set $\osupp(D)$, where 
$$
D = \supp(\lambda-\lambda_0) \cup \supp(\mu-\mu_0) \cup \supp(\rho- \rho_0),
$$
as we did in \cite{EP24}, with a linearized method. This results in a significant reduction in 
the number of computations needed to do the reconstruction, and is therefore of great interest
from a computational point of view.
The linearized method presented here is thus an improvement of the method in
\cite{EP24} in terms of computation time. The linearized method is drastically faster to compute,
which is the original motivation for considering the linearized method, see \cite{HU13,HS10}.

Secondly and somewhat surprisingly we show that the linearized method formulated 
here is an improvement on the method of \cite{EP24}  and \cite{EH21}, in that
it allows us to reconstruct additional structure besides $\osupp(D)$. 
We show how to reconstruct 
$$
\osupp(\mu-\mu_0), \quad\text{ and }\quad \osupp(D)\setminus \osupp(\mu-\mu_0), 
$$
as illustrated in figure \ref{fig_intro}.
We can thus determine the shape of the perturbation in $\mu$, modulo internal cavities, and also
which parts of $\osupp(D)$ that do and do not belong to it.
One should however note that we need to assume that $\rho$ is constant to achieve this improvement.
This is nevertheless a significant improvement of the methods in \cite{EP24} and \cite{EH21},
which are only able to recover $\osupp(D)$. Recovery of multiple coefficients of the Navier equation 
is in general a challenging task, and even the uniqueness of the inverse problem remains open.

\begin{figure}[h]
\centering
\includegraphics[scale=0.65]{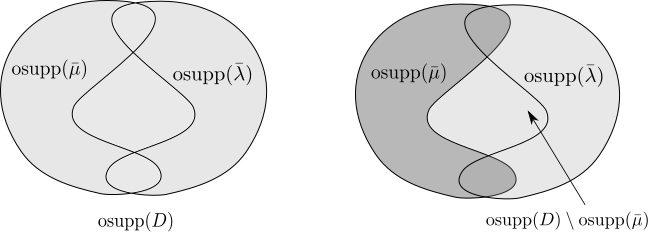}
\caption{
Improvement in the reconstruction. Here $\bar \mu := \mu - \mu_0$ and $\bar \lambda := \lambda - \lambda_0$.
The earlier algorithms in \cite{EP24} and \cite{EH21} essentially reconstruct $\osupp(D)$,
as illustrated on the left hand side.  The linearized method presented here can recover additional structure,
and in particular  the set $\osupp(\bar \mu)$ 
and $\osupp(D) \setminus \osupp(\bar \mu)$, as indicated in the dark grey and light grey regions on the right hand side.}\label{fig_intro}
\end{figure}

\noindent
We also test our algorithm numerically and take a look at two different test
models as well as different kinds of boundary loads. As a first example, we
consider an elastic body with two separated inclusions (see Subsection
\ref{sep_inclu}), where one of the inclusions differs from the background
material only in $\mu$ and the other only in $\lambda$. In the second example,
the test model contains two inclusions that intersect (see Subsection
\ref{int_inclu}). 
Furthermore, we analyse the simulations based on only tangential or only normal boundary loads in
Subsection \ref{diff_boundary}.

\medskip
\noindent
Let us briefly review some of the earlier work related to this problem.
Shape reconstruction  and  related reconstruction methods 
are of interest in geophysical and engineering applications, and
the Navier equation has therefore received a fair amount of attention, 
and several methods have therefore been  studied in this context.
The first works on the shape reconstruction problem for the
Navier equation can be found in \cite{Ar01}, where the linear sampling method is used,
and in \cite{AK02}, which uses the factorization method. 
Reconstruction using the factorization method have been studied in
e.g. \cite{HKS12,HLZ13,EH19}.  Iterative and other  methods are used in \cite{SFHC14,BYZ19,BHQ13,BC05}.
The stationary case $\omega=0$ has also been studied in a number of works.
See e.g. \cite{EHMR21,EH21,II08,Ik99} and for the reconstruction based on data form a lab experiment in \cite{EM21}.
The shape reconstruction problem was originally studied in the context of scalar equations.
For work on electrical  impedance tomography, 
and acoustic and electromagnetic scattering see e.g. \cite{Ta06,Ge08,GK08,Ik90,HPS19b,CDGH20}.

The uniqueness for the inverse problem in isotropic linear elasticity  has been studied extensively.
The earliest work on this include \cite{Ik90} that deals with the linearized problem, and
\cite{ANS91} which deals with boundary determination. 
A significant step towards a global uniqueness result in the three dimensional case was 
given by \cite{NU94} and \cite{ER02}, where it is shown
that the Lamé parameters $\mu$ and $\lambda$ are uniquely determined, by the stationary Dirichlet-to-Neumann map
assuming that $\mu$ is close to a constant.
It should however be noted that the general uniqueness problem is still open, and 
that the current results require some form of smallness assumptions.

Finally let us mention some earlier work relating to the  monotonicity method.
The monotonicity method 
was first formulated in \cite{TR02,Ta06} and also \cite{Ge08,HU13}, for the conductivity equation.
In \cite{HPS19b} and \cite{HPS19a} the method was extended to the analysis of time harmonic waves.
And in \cite{EH21} the monotonicity method was formulated for the recovery of inhomogeneities
in an isotropic elastic body in the stationary case.
The shape reconstruction problem has in the case of  time harmonic waves 
been studied with the monotonicity method in \cite{HPS19b,HPS19a,GH18,Fu20a,Fu20b,AG23}.

\medskip
\noindent
We now turn to discussing some of the main ideas in this work.
When analyzing a vector equation such as \eqref{eq_bvp1} one hopes that the more complex
structure of the vector solutions, as compared to a scalar equation, could be used to analyze the inverse problem
and thus obtain information about both $\lambda$ and $\mu$.
The problem with this is in general that different types of waves tend to couple 
in non constant media and when boundaries are present.
The improved shape reconstruction scheme in Theorem \ref{thm_Lin_inclusionDetection} is inspired by
the fact that elastic waves decouple in a constant background medium into
 pressure wave and shear wave parts, which can be seen by inspecting the Helmholtz decomposition 
of the total wave (see for instance \cite{LL81}).
This suggest that we can create localized solutions with zero divergence provided that
the material parameters are constant.
The derivation of the linearized method  requires the localization of solutions of 
the constant background parameters that are constructed in section \ref{sec_loc2}.
Dealing with constant coefficients greatly simplifies the construction of zero divergence solutions,
since we can utilize the split of the solution into a shear wave part with no divergence, and
a pressure wave part with divergence.
We translate the problem of finding zero divergence solutions to the constant coefficient 
Navier equations to finding suitable solutions of a vector Helmholtz equation,
which we in turn analyze using a second order system of Maxwell type.
A complication in the construction of the localized solutions comes from the fact that we are dealing
with time harmonic waves and the monotonicity inequality of Lemma \ref{lem_monotonicity_ineq1} only holds
in the orthogonal complement of a finite dimensional space $V \subset L^2(\Gamma_N)^3$.
We need in subsection \ref{sec_zero_div} to translate the finite dimensional constraint on the 
boundary condition $g$ of equation \ref{eq_bvp1}, which is of  the form $g \perp V$,
to a condition on the boundary value $f$ for the vector Helmholtz equation \eqref{eq_curlBVP}.
This condition will be of the form $f \in W_1$, where $W_1$ has finite codimension. 

The linearized monotonicity method is based on computing the Fréchet derivative of the  Neumann-to-Dirichlet operator. 
We like to note that a different approach is needed for evaluating the Fréchet derivative here, than in the stationary
case $\omega=0$ in \cite{EH21}. This is due to that equation
\eqref{eq_bvp1} includes the zero order density term $\rho$.
We are unable to use an explicit monotonicity inequality as in \cite{EH21}, to compute the Fréchet derivative.
We utilize instead elliptic estimates. See Proposition \eqref{prop_elipEst_1}.
The proof of the linearized shape reconstruction method of Theorem \ref{thm_Lin_inclusionDetection} differs 
also in this respect from \cite{EH21}. Here we again also need to use elliptic estimates, instead of solely relying on 
explicit monotonicity type inequalities, due to the zero order term in the equation \eqref{eq_bvp1}. 
Because of this we obtain less explicit constants in Theorem \ref{thm_Lin_inclusionDetection} than in the
corresponding result in \cite{EH21}.

Finally we like to point out that there is 
a major obstacle in recovering multiple coefficients from a single frequency $\omega$.
Unique determination of the parameters by $\Lambda$ 
is expected to fail  when a density term $\rho$ is added to the stationary equation,
and one expects there to be several distinct tuples of coefficients 
that match the given boundary measurements.
This is well understood in the scalar case and 
a straight forward mechanism for obtaining counter examples to uniqueness in this case
was given in \cite{AL98}.
This phenomenon affects in more general also vector valued equations,
and thus a  similar form of non-uniqueness should  effect  the corresponding inverse problem for the Navier equation.
It is worth to note that this form of non-uniqueness is related to the existence of certain forms
of gauge transformations that transform the material parameters, but leaves the boundary data alone, see
\cite{IS23}.
We note that our shape reconstruction method works in the general case
only for inhomogeneities where the density either remains constant or becomes 
smaller than the background density. 

One should also note that knowing $\Lambda$ for multiple $\omega$ 
could possibly remove the non-uniqueness when  trying to recover several of  the coefficients in \eqref{eq_bvp1}.
See \cite{Na88}. One can also impose additional conditions on the coefficients, such as 
piecewise analyticity 
in order to recover several coefficients using a single frequency, see \cite{Ha09}.
This is explored in the case of the Navier equation and  piecewise constant coefficients   
in \cite{BHFVZ17}.

\medskip
\noindent
The paper is structured as follows. Section \ref{sec_defs} contains some definitions and preliminaries that 
are used through out the paper. 
After this we compute the Fréchet derivative of the ND-map in section \ref{sec_frechet}.
In sections \ref{sec_loc1} and \ref{sec_loc2} we construct localized solutions
in particular for the constant material parameter case, with a finite dimensional constraint.
In sections \ref{sec_mono_shape}, \ref{sec_reconstruct_D} and \ref{sec_mu_recons} we 
formulate and derive the linearized reconstruction procedure. 
Section \ref{sec_numerics} contains numerical tests. The appendix in section \ref{sec_ellipticEst}
contains elliptic estimates on the source problem that we need in the earlier sections.
Section \ref{sec_maxwell} in the appendix finally deals with the Fredholm theory of a second order system of Maxwell type.

\section{Definitions and  other preliminaries} \label{sec_defs}

\noindent
In this section we review some preliminaries that will be needed in the rest of the paper.
Our notations are to a large extent the same notations as in \cite{EP24}.
Through out the paper we use the following definitions related to function spaces. 
The space $H^k(\Omega)$ denotes the $L^2(\Omega)$ based Sobolev space with $k$ weak derivatives.
Furthermore we let
$$	
L^\infty_+(\Omega) := \big\{ f \in L^\infty(\Omega) \;:\; \operatorname{essinf}_\Omega f > 0 \big\}.
$$
The notation $Z^n$  for a function space $Z$
is understood as  $Z^n := Z \times \dots \times Z$, where the right hand side  contains $n$ copies of $Z$.
We denote the  $L^2$-inner product by $(\cdot,\cdot)_{L^2}$, so that 
$$
(u,v)_{L^2(\Omega)^{n\times m}} := \int_\Omega u : v\,dx, \quad u,v \in L^2(\Omega)^{n \times m},\quad n,m\geq 1,
$$
where $:$ is the Frobenius inner product defined below.
For orthogonality with respect to the $L^2$-inner product, we use the notation $\perp$, unless otherwise stated,
so that 
$$
u \perp v \qquad \Leftrightarrow \qquad (u,v)_{L^2(\Omega)^n} = 0, \quad  \text{ when } u,v \in L^2(\Omega)^n.
$$

We will now consider to the well-posedness of problem \eqref{eq_bvp1}.
The bilinear form $B$ related to equation \eqref{eq_bvp1} is given by
\begin{align}  \label{eq_weak}
B(u,v)  := -\int_\Omega 
2 \mu \hat \nabla u :\hat \nabla v + \lambda \nabla \cdot u \nabla \cdot v - \omega^2\rho u\cdot v\,dx, 
\end{align}
for all $u,v \in H^1(\Omega)^3$. Where the  Frobenius inner product $A:B$ is defined as
$$
A:B = \sum A_{ij}B_{ij},  \qquad A,B \in \R^{m\times n}.
$$
Note that the Euclidean norm on $\R^{m \times n}$, $m,n \in \N$, is given by
$|A| = (A:A)^{1/2}$, for $A \in \R^{m \times n}$.
We will use the notation
$$
L_{\lambda,\mu,\rho} u := \nabla \cdot ( \C \hat \nabla u)+ \omega^2\rho u. 
$$
Note that in an isotropic medium characterized by the Lamé parameters the above 
equation simplifies to
$$
L_{\lambda,\mu,\rho} u = \nabla \cdot ( 2 \mu \hat \nabla u + \lambda (\nabla \cdot u) I ) + \omega^2\rho u.
$$
A weak solution to \eqref{eq_bvp1} is defined as a $u \in H^1(\Omega)^3$,
for which $u|_{\Gamma_D} = 0$ and 
\begin{align}  \label{eq_weak2}
B(u,v)  = - \int_{\Gamma_N} g\cdot v \,dS,
\qquad \forall v \in \{ u \in H^1(\Omega)^3 \,:\, u|_{\Gamma_D} = 0 \}.
\end{align}
For the existence and uniqueness of a weak solution to \eqref{eq_bvp1}, when $\omega$
is not a resonance frequency see Corollary 3.4 in \cite{EP24}.

The existence and uniqueness of a weak solution to \eqref{eq_bvp1} implies that 
the Neumann-to-Dirichlet map $\Lambda$ given by \eqref{eq_ND_map} is well defined. $\Lambda$ is related to $B$ as  follows.
When the material parameters are regular, and $u$ solves \eqref{eq_bvp1} with $g$, and
$v$ solves \eqref{eq_bvp1} with $h$, we see by integrating by parts that
\begin{align}  \label{eq_NDmap}
B(u,v)  = -\int_{ \Gamma_N} g \cdot v \,dS = -( g, \Lambda h)_{L^2(\Gamma_N)^3}. 
\end{align}
Furthermore we abbreviate the  boundary condition in \eqref{eq_bvp1} by
\begin{align*} 
\gamma_{\mathbb{C},\Gamma} u  = (\C \, \hat \nabla u ) \nu |_{\Gamma},
\end{align*}
or with $\gamma_\mathbb{\C} u $ if the boundary is clear from the context. 
Note that these notations are formal when $u \in H^1(\Omega)^3$, since we cannot in general take the trace of an 
$L^2(\Omega)^3$ function. In the low regularity case we understand the boundary condition in a weak sense.
And define $\gamma_{\mathbb{C}} u \in L^2(\Gamma_N)^3$, 
for $u \in H^1(\Omega)^3$ which solve \eqref{eq_bvp1}, as the element in $L^2(\Gamma_N)^3$, for which
\begin{align*} 
-(\gamma_{\mathbb{C}} u, \, \varphi|_{\Gamma_N}  )_{L^2(\Gamma_N)^3}
= B(u,\varphi), \qquad \forall \varphi \in H^1(\Omega)^3. 
\end{align*}
For the convenience of the reader we also list the following integration by parts formulas
that are used through out the paper. Firstly we have the matrix form of the divergence theorem
$$ 
\int_\Omega \nabla \cdot  A  u \,dx = - \int_\Omega A:\hat \nabla u \,dx + \int_{\p \Omega} A \nu \cdot u\,dS,
\quad A \in H^1(\Omega)^{n\times n},
$$
where $u,v \in H^1(\Omega)^3$. See \cite{Ci88} p. xxix. We also use the integration by parts  formula
\begin{align*} %\label{eq_intbyparts1}
\int_\Omega \nabla \times u \cdot v \,dx = \int_\Omega  u \cdot \nabla \times v \,dx
+
\int_{\p \Omega} (\nu \times u) \cdot v\,dS, \quad u,v \in H^1(\Omega)^3.
\end{align*}
Next we give a definition of the notion of outer support of a function or a set.
The outer support (with respect to $\p\Omega$) of a measurable function $f: \Omega \to \R$   is defined as 
$$
\operatorname{osupp} (f) := \Omega \setminus \bigcup \, \big \{  U \subset \Omega \,:\, U  
\text{ is relatively open and connected to $\p \Omega$},
f|_U \equiv 0  \big\}.
$$
For more on this see \cite{HU13}.
It will be convenient to extend this definition to sets. We define the outer support of a measurable
set $D \subset \Omega$ 
(with respect to $\p\Omega$) as
\begin{align}  \label{eq_def_osupp}
\osupp (D) := \osupp(\chi_D)
\end{align}
where $\chi_D$ is the characteristic function of the set $D$. 

\medskip
\noindent
We will need the monotonicity inequality of Proposition 4.1 in \cite{EP24}, which is given by the following lemma.

\begin{lem} \label{lem_monotonicity_ineq1}
Let $\mu_j,\lambda_j,\rho_j \in L^\infty_+(\Omega)$, for $j=1,2$ and $ \omega \neq 0$.
Let $u_j$ denote the solution to \eqref{eq_bvp1} where $\mu=\mu_j, \lambda= \lambda_j$ and $ \rho=\rho_j$,
with the boundary value  $g$. There exists a finite dimensional subspace $V\subset L^2( \Gamma_N)^3$, 
such that
\begin{align*}  %\label{eq_mono1}
\big(  (\Lambda_2 - \Lambda_1)g, \,g  \big )_{L^2(\Gamma_N)^3} 
\geq
\int_\Omega 
2(\mu_1-\mu_2 ) |\hat \nabla u_1 |^2 + (\lambda_1 - \lambda_2 ) |\nabla \cdot u_1|^2  + \omega^2(\rho_2-\rho_1) |u_1|^2 \,dx, 
\end{align*}
when $g \in V^\perp$. %We have moreover that $\dim(V) \leq d(\lambda_2,\mu_2,\rho_2)$.
\end{lem}

%%%%%%%%%%%%%%%%%%%%%%%%%%%%%%%%%%%%%%%%%%%%%%%%%%%%%%%%%%%%%%%%%%%%%%%%%%%%%%%%%%%%%%%%%%%%%%%%%%%%%%%%%%%%%%%

\section{The Fréchet derivative of the Neumann-to-Dirichlet map} \label{sec_frechet}
\noindent
Before we take a look at the Fréchet derivative of the Neumann-to-Dirichlet
map, we need some estimates which are given in the following lemmas.
\begin{lem}\label{lem_ab_1}
Let $\mu_j,\lambda_j,\rho_j \in L^\infty_+(\Omega)$, for $j=1,2$ and $ \omega \neq 0$.
Let $u_j$ denote the solution to \eqref{eq_bvp1} where $\mu=\mu_j, \lambda= \lambda_j$ and $ \rho=\rho_j$,
with the boundary value $g$. Then
\begin{align*}
\left ((\Lambda_2  - \Lambda_1) g, g \right )_{L^2(\Gamma_{\textup{N}})^3}
\leq &\int_\Omega2(\mu_1-\mu_2)|\hat\nabla u_2|^2+(\lambda_1-\lambda_2)|\nabla\cdot u_2|^2+\omega^2(\rho_2-\rho_1)|u_2|^2dx \\
&+\omega^2\int_\Omega\rho_1|u_1-u_2|^2dx .
\end{align*}
\end{lem}

\begin{proof}
Since $\Lambda_2 g=u_2|_{\Gamma_{\textup{N}}}$, we can use the variational formula for $\Lambda_1$ and $\Lambda_2$ with $v=u_2$ and obtain
\begin{align*}
&\int_\Omega2\mu_1(\hat\nabla u_1:\hat\nabla u_2)+\lambda_1(\nabla\cdot u_1)(\nabla\cdot u_2)dx-\omega^2\int_\Omega\rho_1u_1\cdot u_2dx\\
&=\left(\Lambda_2 g,g\right )_{L^2(\Gamma_{\textup{N}})^3}\\
&=\int_{\Gamma_{\textup{N}}}g\cdot u_2 ds\\
&=\int_\Omega2\mu_2(\hat\nabla u_2:\hat\nabla u_2)+\lambda_2(\nabla\cdot u_2)(\nabla\cdot u_2)dx-\omega^2\int_\Omega\rho_2u_2\cdot u_2dx .
\end{align*}
\noindent
Thus
\begin{align*}
&\int_\Omega2\mu_1|\hat\nabla (u_1-u_2)|^2+\lambda_1|\nabla\cdot (u_1-u_2)|^2dx- \omega^2\int_\Omega\rho_1|u_1-u_2|^2dx\\
&\quad=\int_\Omega2\mu_1|\hat\nabla u_1|^2+\lambda_1|\nabla\cdot u_1|^2dx- \omega^2\int_\Omega\rho_1|u_1|^2dx\\
&\qquad+\int_\Omega2\mu_1|\hat\nabla u_2|^2+\lambda_1|\nabla\cdot u_2|^2dx- \omega^2\int_\Omega\rho_1|u_2|^2dx\\
&\qquad-2\left(\int_\Omega2\mu_1(\hat\nabla u_1):(\hat\nabla u_2)+\lambda_1(\nabla\cdot u_1)(\nabla\cdot u_2)dx- \omega^2\int_\Omega\rho_1u_1\cdot u_2dx\right)\\
&=\quad\left( \Lambda_1 g, g \right )_{L^2(\Gamma_{\textup{N}})^3}-\left(\Lambda_2 g, g \right )_{L^2(\Gamma_{\textup{N}})^3}\\
&\qquad+\int_\Omega2(\mu_1-\mu_2)|\hat\nabla u_2|^2+(\lambda_1-\lambda_2)|\nabla\cdot u_2|^2dx- \omega^2\int_\Omega(\rho_1-\rho_2)|u_2|^2dx
\end{align*}
Hence,
\begin{align*}
0&\leq\int_\Omega2\mu_1|\hat\nabla (u_1-u_2)|^2+\lambda_1|\nabla\cdot (u_1-u_2)|^2dx\\
&=\left ( \Lambda_1 g, g \right )_{L^2(\Gamma_{\textup{N}})^3} -\left(\Lambda_2 g, g \right )_{L^2(\Gamma_{\textup{N}})^3}\\
&\qquad+\int_\Omega2(\mu_1-\mu_2)|\hat\nabla u_2|^2+(\lambda_1-\lambda_2)|\nabla\cdot u_2|^2dx- \omega^2\int_\Omega(\rho_1-\rho_2)|u_2|^2dx\\
&\qquad+ \omega^2\int_\Omega\rho_1|u_1-u_2|^2dx,
\end{align*}
so that
\begin{align*}
\left(\Lambda_2 g, g \right )_{L^2(\Gamma_{\textup{N}})^3} -\left( \Lambda_1 g, g \right )_{L^2(\Gamma_{\textup{N}})^3}
&\leq\int_\Omega2(\mu_1-\mu_2)|\hat\nabla u_2|^2+(\lambda_1-\lambda_2)|\nabla\cdot u_2|^2dx\\
&\qquad+ \omega^2\int_\Omega(\rho_2-\rho_1)|u_2|^2dx+ \omega^2\int_\Omega\rho_1|u_1-u_2|^2dx.
\end{align*}
\end{proof}

\begin{lem} \label{lem_w_source}
Let $h=(h_\lambda,h_\mu,h_\rho) \in L^\infty(\Omega)^3$, and 
let $\mu_1=\mu_0+h_\mu$, $\lambda_1=\lambda_0+h_\lambda$ and $\rho_1=\rho_0+h_\rho$. 
Furthermore let $u_1\in H^1(\Omega)^3$ and $u_0 \in H^1(\Omega)^3$ be the solution of the
boundary value problem \eqref{eq_bvp1} for $(\rho_1,\lambda_1,\mu_1)$ and
$(\rho_0,\lambda_0,\mu_0)$, respectively.
The difference $u_1-u_0$  is a weak solution to 
\begin {align}\label{L_1}
\begin{cases}
L_{\lambda_1 , \mu_1 ,\rho_1 } (u_1-u_0) &=-(\nabla\cdot(2h_\mu\hat\nabla u_0+h_\lambda\nabla\cdot u_0)-\omega^2 h_\rho u_0), \\
\gamma_{\C_1} (u_1 - u_0)|_{\Gamma_N} &= \gamma_{\C_{h_\lambda,h_\mu}} u_0|_{\Gamma_N},\\ 
(u_1 - u_0)|_{\Gamma_D} &= 0,
\end{cases}
\end{align}
where $\C_{h_\lambda,h_\mu}$ is given by \eqref{eq_Cdef}. %where we set $\lambda = h_\lambda$ and $\mu = h_\mu$.
\end{lem}
\noindent
\begin{proof} 
The $u_j$ are weak solutions to 
\begin{align*}
\nabla\cdot(\C_j\hat\nabla u_j)+ \omega^2\rho_ju_j&=0, \\
(\C_j\hat\nabla u_j)\nu|_{\Gamma_N} &=g,
\end{align*}
where  $j=0,1$ means that
$$
B_j(u_j,\varphi) = -\int_{\Gamma_N} g \cdot \varphi \,dS, 
$$
where $\varphi$ is  a test function. Observe that in general for a $v \in H^1(\Omega)^3$, 
$$
B_1(v , \varphi) = 
B_0(v, \varphi) + 
\int_\Omega 2 h_\mu \hat \nabla v :\hat \nabla \varphi + h_\lambda \nabla \cdot v \nabla \cdot \varphi - \omega^2h_\rho v\cdot \varphi \,dx. 
$$
Using the weak formulations of the above problems and this we obtain that
\begin{align*}
B_1(u_1 - u_0, \varphi)  =   
-\int_\Omega 2 h_\mu \hat \nabla u_0 :\hat \nabla \varphi + h_\lambda \nabla \cdot u_0 \nabla \cdot \varphi 
- \omega^2h_\rho u_0\cdot \varphi \,dx.
\end{align*}
This is an instance of the weak
formulation \eqref{eq_weak3} of \eqref{eq_bvp2}, which proves the claim.
\end{proof}

\begin{lem}\label{lem_ab_2}
Let $u_1\in H^1(\Omega)^3$ and $u_0 \in H^1(\Omega)^3$ be as in Lemma \ref{lem_w_source}.  Then
\begin{align*}
\|u_1-u_0\|_{H^1(\Omega)^3}\leq\mathcal{O}(\|h\|_{\infty})
\end{align*}
\noindent
with $h=(h_{\rho},h_{\lambda},h_{\mu})$.
\end{lem}
\begin{proof}
By interpreting the left hand side of \eqref{L_1} in a weak sense we see that $u_1-u_0$ solves \eqref{eq_bvp2}, with
$$
\int_\Omega F \cdot v \,dx - \int_{\Omega}  A : \hat \nabla v \,dx
=
\int_{\Omega}2 h_\mu \hat{\nabla} u_0:\hat{\nabla}v + h_\lambda \nabla\cdot u_0 \nabla\cdot v - \omega^2 h_\rho u_0 v\,dx.
$$
for $v \in \mathcal{V}$. Moreover by the Cauchy-Schwarz inequality we have the estimate
\begin{small}
\begin{align*}
\| F \|_{L^2( \Omega)^{3}}+\| A \|_{L^2( \Omega)^{3 \times 3}}
&\leq C(
\| h_\mu \|_{ L^\infty}  \| \hat \nabla v \|_{L^2(D_2)^{3 \times 3}}
+
\| h_\lambda \|_{ L^\infty}  \| \nabla \cdot v \|_{L^2(D_2)}
+
\| h_\rho \|_{ L^\infty}  \|  v \|_{L^2(D_2)^{3}}) \\
&\leq
\mathcal{O}( \| h \|_{L^\infty}).
\end{align*}
\end{small}
The claim follows now from estimate \eqref{eq_aprioriEst2} and the assumption that $\omega$ is not a resonance frequency.
\end{proof}

\begin{lem}\label{lem_ab_3}
Let $u_1\in H^1(\Omega)^3$ and $u_0 \in H^1(\Omega)^3$ be as in Lemma \ref{lem_w_source}. 
% For $\|h\|_\infty\rightarrow0$ we have
We have that
\begin{align*}
\|u_1+u_0\|_{L^2(\Omega)^3} &\leq 2\|u_0\|_{L^2(\Omega)^3} +  \mathcal{O}( \| h \|_{L^\infty}). %C_1.
\end{align*}
\noindent
\end{lem}

\noindent
\begin{proof}
We have
\begin{align*}
\|u_1+u_0\|_{L^2(\Omega)^3}\leq\|u_1-u_0\|_{L^2(\Omega)^3}+2 \|u_0\|_{L^2(\Omega)^3}.
\end{align*}
\noindent
From Lemma \ref{lem_ab_2} it follows that
\begin{align*}
\|u_1-u_0\|_{L^2(\Omega)^3}\leq\|u_1-u_0\|_{H^1(\Omega)^3}\leq  \mathcal{O}(\| h \|_{L^\infty}), %\dfrac{C}{\beta}\|u_0\|_{H^1(\Omega)^3} \|h\|_{\infty}.
\end{align*}
\noindent
which proves the claim.
\end{proof}

\noindent

\begin{lem} \label{lem_Frechet}
Let $u_g$ and $u_f$ be the solution to (\ref{eq_bvp1}) for the boundary loads $g$ and $f$, respectively.
There exists a bounded linear operator $\Lambda^\prime_{\lambda,\mu,\rho}$ such that
\begin{align*} 
\lim_{\| h\|_{\infty} \to 0}
\dfrac{1}{\| h\|_{\infty}} 
\big\|\Lambda_{\lambda+h_\lambda,\mu+h_\mu,\rho+h_\rho}-\Lambda_{\lambda,\mu,\rho}
-\Lambda^\prime_{\lambda,\mu,\rho}[h_\lambda,h_\mu,h_\rho] \big\|_{*}
= 0
\end{align*} 
where $h=(h_\rho,h_\lambda,h_\mu) \in L^\infty(\Omega)^3$, $ \| \cdot \|_{*}$ is the operator norm, and
where the associated bilinear form of the Fréchet derivative is
\begin{align}\label{bilinear_Frechet}
&\left( \Lambda^\prime_{\lambda,\mu,\rho}[h_\lambda,h_\mu,h_\rho] g,f \right )_{L^2(\Gamma_{\textup{N}})^3}
=-\int_{\Omega} 2 h_\mu \hat{\nabla}u_g : \hat{\nabla} u_f + h_\lambda \nabla \cdot u_g \nabla \cdot u_f - \omega^2 h_\rho u_g \cdot u_f \,dx.
\end{align}
\end{lem}

\begin{proof}
By symmetric bilinearity estimates, we obtain 
\begin{align*}
&\|\Lambda_{\lambda+h_\lambda,\mu+h_\mu,\rho+h_\rho}-\Lambda_{\lambda,\mu,\rho}
-\Lambda^\prime_{\lambda,\mu,\rho}[h_\lambda,h_\mu,h_\rho]\|_{*}\\
% &=\sup_{\|g\|_{L^2(\Gamma_\textup{N})^3}=\|f\|_{L^2(\Gamma_\textup{N})^3}=1}
&=\sup_{\|g\|_{L^2}=\|f\|_{L^2}=1}
\vert \left( (\Lambda_{\lambda+h_\lambda,\mu+h_\mu,\rho+h_\rho}
-\Lambda_{\lambda,\mu,\rho}-\Lambda^\prime_{\lambda,\mu,\rho}[h_\lambda,h_\mu,h_\rho])g,f \right )_{L^2(\Gamma_{\textup{N}})^3}\vert\\
& \leq
% C_1 \sup_{\|g\|_{L^2(\Gamma_\textup{N})^3}=1}
C \sup_{\|g\|_{L^2}=1}
\vert \underbrace{\left( (\Lambda_{\lambda+h_\lambda,\mu+h_\mu,\rho+h_\rho}
-\Lambda_{\lambda,\mu,\rho}-\Lambda^\prime_{\lambda,\mu,\rho}[h_\lambda,h_\mu,h_\rho])g,g \right )_{L^2(\Gamma_{\textup{N}})^3}}_{:=F}\vert
\end{align*}
and find with Lemma \ref{lem_ab_1} and Lemma \ref{lem_ab_2} for 
\begin{align*}
(\lambda_1, \mu_1, \rho_1) := (\lambda+h_\lambda, \mu+h_\mu,\rho+h_\rho),
\quad \text{and} \quad
(\lambda_2,\mu_2,\rho_2) := (\lambda, \mu,\rho) 
\end{align*}
where $u$ is the solution for the Lam\'{e} parameters $(\lambda, \mu,\rho )$ and $u_1$ the solution for the Lam\'{e}
parameters $(\lambda+h_\lambda, \mu+h_\mu,\rho+h_\rho)$ that
\begin{align*}
F\geq &\int_\Omega2(\mu-\mu-h_\mu)|\hat\nabla
u|^2+(\lambda-\lambda-h_\lambda)|\nabla\cdot u|^2dx+ \omega^2
\int_\Omega(\rho+h_\rho-\rho)|u|^2dx\\
&- \omega^2\int_\Omega(\rho+h_\rho)|u_1-u|^2dx-\left(-\int_{\Omega} 2 h_\mu
|\hat{\nabla}u|^2 + h_\lambda |\nabla \cdot u |^2 - \omega^2 h_\rho |u|^2
\,dx\right)\\
&=-\omega^2 \int_\Omega(\rho+h_\rho)|u_1-u|^2dx\\
&=\mathcal{O}(\|h\|_\infty^2).
\end{align*}
On the other hand, we have with Lemma \ref{lem_ab_1}
\begin{align*}
F& \leq\int_\Omega2(\mu-\mu-h_\mu)|\hat\nabla
u_1|^2+(\lambda-\lambda-h_\lambda)|\nabla\cdot u_1|^2dx+ \omega^2
\int_\Omega(\rho+h_\rho-\rho)|u_1|^2dx\\
&\qquad+ \omega^2 \int_\Omega\rho|u_1-u|^2dx-\left(-\int_{\Omega} 2 h_\mu
|\hat{\nabla}u|^2 + h_\lambda |\nabla \cdot u |^2 - \omega^2 h_\rho |u|^2
\,dx\right)\\
% \displaybreak
&=\int_\Omega2h_\mu(|\hat\nabla u|^2-|\hat\nabla u_1|^2)+h_\lambda(|\nabla\cdot
u|^2-|\nabla\cdot u_1|^2)dx+ \omega^2 \int_\Omega h_\rho(|u_1|^2-|u|^2)dx\\
&\qquad+  \omega^2\int_\Omega\rho|u_1-u|^2dx\\
&=\int_\Omega2h_\mu\hat\nabla (u-u_1):\hat\nabla (u+u_1)+h_\lambda\nabla\cdot (u-u_1)\nabla\cdot (u+u_1)dx\\
&\qquad+\omega^2 \int_\Omega h_\rho(u_1-u)\cdot(u_1+u)dx+\omega^2\int_\Omega\rho|u_1-u|^2dx\\
%&\leq\int_\Omega2h_\mu(|\hat\nabla (u-u_1):\hat\nabla (u+u_1))|+h_\lambda(|\nabla\cdot (u-u_1)\nabla\cdot (u+u_1)|)dx\\
%&\qquad+\omega^2\int_\Omega h_\rho(|(u_1-u)\cdot(u_1+u)|)dx+\omega^2\int_\Omega\rho|u_1-u|^2dx\\
&\leq\int_\Omega2h_\mu|\hat\nabla (u-u_1)|\,\,|\hat\nabla (u+u_1)|+h_\lambda|\nabla\cdot (u-u_1)|\,\,|\nabla\cdot (u+u_1)|dx\\
&\qquad+\omega^2\int_\Omega h_\rho|u_1-u|\,\,|u_1+u|dx+\omega^2\int_\Omega\rho\underbrace{|u_1-u|^2}_{=\mathcal{O}(\|h\|_\infty^2)}dx\\
&\leq C \|h\|_\infty \|u-u_1\|_{H^1(\Omega)^{3}} \|u+u_1\|_{H^1(\Omega)^{3}} +\mathcal{O}(\|h\|_\infty^2)
\end{align*}
Since $u,u_1\in H^1(\Omega)$, $u_1-u\in\mathcal{O}(\|h\|_\infty)$ due to Lemma \ref{lem_ab_2} and \ref{lem_ab_3} and the fact, that we are only interested in the limit $\|h\|_\infty\rightarrow0$, we can estimate
\begin{align*}
\|u_1+u\|_{H^1}&\leq2\|u\|_{H^1}+\mathcal{O}( \| h \|_{L^\infty}),
\end{align*}
and
\begin{align*}
\|\hat{\nabla}(u_1-u) \|_{L^2(\Omega)^{3\times3}}&\leq C\|u_1-u\|_{H^1(\Omega)^3}=\mathcal{O}(\|h\|_\infty)\\
\|\nabla\cdot(u_1-u)  \|_{L^2(\Omega)}&\leq C \|u_1-u\|_{H^1(\Omega)^3}=\mathcal{O}(\|h\|_\infty)
\end{align*}
Hence,
\begin{align*}
F\leq C\|h\|_\infty \|u-u_1\|_{H^1(\Omega)^{3}} \|u+u_1\|_{H^1(\Omega)^{3}} +\mathcal{O}(\|h\|^2_{\infty})
=\mathcal{O}(\|h\|_\infty^2)
\end{align*}
We have to show that the operator is bounded and linear and take a look at the associated bilinear form:
\begin{align*}
\left(\Lambda^\prime_{\lambda,\mu,\rho}[\hat{\lambda},\hat{\mu},\hat{\rho},] g,f \right)_{L^2(\Gamma_{\textup{N}})^3}
= -\int_{\Omega}2\hat{\mu} \hat{\nabla}u_g :\hat{\nabla}u_f +\hat{\lambda}
\nabla\cdot u_g \nabla\cdot u_f - \omega^2 \hat{\rho} u_g \cdot u_f\,dx.
\end{align*}
\noindent
First, we prove the boundedness and consider
\begin{align*}
 \left( \Lambda^\prime_{\lambda,\mu,\rho}[ h_\lambda, h_\mu,h_\rho] g,f \right)_{L^2(\Gamma_{\textup{N}})^3} 
&= -\int_{\Omega}2 h_\mu \hat{\nabla}u_g :\hat{\nabla}u_f +h_\lambda\nabla\cdot u_g \nabla\cdot u_f - \omega^2 h_\rho u_g \cdot u_f\,dx\\
&\leq \beta \|h\|_\infty \|u_g \|_{H^1(\Omega)^3} \|u_f \|_{H^1(\Omega)^3}.
\end{align*}
\noindent
Clearly, we have linearity.
\noindent
All in all, this leads to the desired result.
\end{proof}

\begin{lem} \label{lem_compact_self_adjoint}
The Fréchet derivative $\Lambda_{\lambda,\mu,\rho}'$ is compact and self-adjoint.
\end{lem}
\noindent
\begin{proof}
$\Lambda_{\lambda,\mu,\rho}^\prime$ is obviously symmetric since
\begin{align*}
\left( \Lambda_{\lambda,\mu,\rho}^\prime[\hat{\lambda},\hat{\mu},\hat{\rho}]g,f \right)_{L^2(\Gamma_{\textup{N}})^3}
&=-\int_{\Omega}2\hat{\mu}\hat{\nabla}u_g:\hat{\nabla}u_f +\hat{\lambda}(\nabla\cdot u_g)(\nabla \cdot u_f) -\omega^2 \hat{\rho} u_g\cdot u_f\,dx\\
&=\left( g,\Lambda_{\lambda,\mu,\rho}^\prime[\hat{\lambda},\hat{\mu},\hat{\rho}]f \right)_{L^2(\Gamma_{\textup{N}})^3}
\end{align*}
\noindent
due to the symmetry to the bilinear form. Since $\Lambda_{\lambda,\mu,\rho}'$ is bounded it is therefore also self-adjoint. 
\\
\\
Next we prove that $\Lambda_{\lambda,\mu,\rho}'$ is compact. By the formula in Lemma	\ref{lem_Frechet}, we have that
\begin{align*}
( \Lambda^\prime_{\lambda,\mu,\rho}[h_\lambda,h_\mu,h_\rho] g,f  )_{L^2(\Gamma_{\textup{N}})^3}
& =-\int_{\Omega} 2 h_\mu \hat{\nabla}u_g : \hat{\nabla} u_f + h_\lambda \nabla \cdot u_g \nabla \cdot u_f - \omega^2 h_\rho u_g \cdot u_f \,dx\\
&=: -\int_{\Omega} A : \hat \nabla u_f - F \cdot u_f \,dx. 
\end{align*}
Since $\omega$ is a non-resonance frequency, we have a unique weak solution $v_g$ to
\begin{align*}
\begin{cases}
\nabla \cdot (\C\,  \hat \nabla v_g )  + \omega^2\rho v_g  &=  F + \nabla \cdot A, \\
% \sum_{k=1}^3 \p_k F^k \\ %\phi_1F^1 + \nabla \cdot (\phi_2 (\nabla \cdot F^2)I)+\nabla\cdot (\phi_3 \hat \nabla F^3), \\
\;\quad\quad\quad\quad(\gamma_\C v_g ) |_{\Gamma_N} &=  A \nu |_{\Gamma_N} , \\	
 \quad\quad\quad\quad\quad \quad v_g |_{\Gamma_D} &= 0,	
\end{cases}
\end{align*}
by Proposition \ref{prop_elipEst_1}, when $\C$ is given by $\lambda$ and $\mu$.
Thus by \eqref{eq_weak3} we have in particular that
$$
B(v_g,v) = \int_\Omega F \cdot v \,dx - \int_{\Omega}  A : \hat \nabla v \,dx. 
$$
Since $u_f$ is  a solution of \eqref{eq_bvp1}, we have that
$$
( \Lambda^\prime_{\lambda,\mu,\rho}[h_\lambda,h_\mu,h_\rho] g,f  )_{L^2(\Gamma_{\textup{N}})^3}
= B(v_g,u_f)  =  -\int_{\Gamma_N} v_g \cdot f \, dS. 
$$
Hence we have that
$$
\Lambda^\prime_{\lambda,\mu,\rho}[h_\lambda,h_\mu,h_\rho] g = -v_g |_{\Gamma_N}.
$$
Since $v_g \in H^1(\Omega)^3$, $v_g |_{\Gamma_N} \in H^{1/2}(\Gamma_N)^3$.
Compactness now follows from the compactness of the inclusion $H^{1/2}(\Gamma_N)^3 \hookrightarrow L^2(\Gamma_N)^3$.
\end{proof}

\section{Localization of solutions} \label{sec_loc1}

\noindent
To derive the monotonicity based shape reconstruction procedure one uses
localized solutions.
In this section we briefly review the results on localizing solutions for the Navier equation
that were obtained in \cite{EP24}. In the next section we will consider localization
in the case of constant material parameters.

\medskip
\noindent
Let $D_1,D_2 \subset \Omega$.
We will localize a solution so that it is small in $D_1 \subset \Omega$
and large in  $D_2 \subset \Omega$. We will assume that $\p D_1$ is Lipschitz and that $\p D_2$ is smooth, and moreover that
\begin{align}  \label{eq_Dassump}
D_1 \cap D_2 = \emptyset, \qquad \Omega \setminus (D_1 \cup D_2) \text{ is connected }, \qquad
\overline{\Omega \setminus (D_1 \cup D_2)}\cap\Gamma_N \neq \emptyset.
\end{align}

\noindent In \cite{EP24} we proved the following localization result for 
solutions that are guaranteed to have non-zero divergence in $D_2$.

\begin{prop} \label{prop_loc2}
Assume that $D_1,D_2 \subset \Omega$ are  as in \eqref{eq_Dassump}, and that $D'_i \Subset D_i$, $i=1,2$
are open and non-empty.
Let $V\subset L^2(\Gamma_N)^3$ be a subspace with $\dim (V)<\infty$,
then there exists a sequence $g_j \in L^2(\Gamma_N)^3$, such that $ g_j \perp V$ in the $L^2(\Gamma_N)^3$-norm,
and for which
$$
\| u_j \|_{ L^2(D_1)^3 }, 
\;\| \hat \nabla u_j \|_{ L^2(D'_1)^{3\times 3} }, 
\;\| \nabla \cdot u_j \|_{ L^2(D'_1)} 
\to 0, 
$$
and for which
$$
\| u_j \|_{ L^2(D_2)^3 }, 
\;\| \hat \nabla u_j \|_{ L^2(D'_2)^{3\times 3} }, 
\;\| \nabla \cdot u_j \|_{ L^2(D'_2)} 
\to \infty, 
$$
as $j \to \infty$, and where $u_j$ solves \eqref{eq_bvp1}, with the boundary conditions $g_j$.
\end{prop}

\section{Localization  of solutions in the case of constant material parameters} \label{sec_loc2}

\noindent
For constant $\lambda_0, \mu_0$ and $\rho_0$ we can in certain respects 
improve on the localisation procedure in \cite{EP24}, which we reviewed in section \ref{prop_loc2}.
Here we show that we can obtain localized solutions with zero divergence.
This is  related to the fact that elastic waves in a constant background medium and in free space
decouple into a pure pressure wave part and a pure separate shear wave part
% \textcolor{blue}{when considered in free space (note that this is not true in general
% when one adds boundaries)}. 
using the Helmholtz decomposition of a vector field.
See e.g. p.101-103 in \cite{LL81}. This type of decoupling does not occur in general
and it is thus important that the material parameters are constant.
\\
\\
The Neumann problem of the  Navier equation \eqref{eq_bvp1}, with $\Gamma_D = \emptyset$,
reduces in the constant coefficient case to
\begin{equation} \label{eq_constNavier}
\begin{aligned}
L_{0} u := \mu_0 \Delta u + (\lambda_0 + \mu_0) \nabla \nabla \cdot u + \omega^2 \rho_0 u &= 0, \quad \text{ in } \Omega, \\
\gamma_{\C_0 }u|_{\p \Omega} &= g.
\end{aligned}
\end{equation}
We will construct a divergence free solution to \eqref{eq_constNavier} 
by constructing a solution $u_s$ to the simpler vector valued Helmholtz  problem
\begin{equation} \label{eq_vecHelmholtz}
\begin{aligned}
\tilde L_0 u_s := \Delta u_s + \omega^2 \frac{\rho_0  }{ \mu_0 }u_s &= 0, \quad \text{ in } \Omega, \\
\nu \times u_s|_{\p \Omega} &= f,
\end{aligned}
\end{equation}
where $f$ is tangential, and $u_s \in H^1(\Omega)^3$.
Note that this solution also solves 
\begin{equation*} %\label{eq_constNavier}
\begin{aligned}
L_0 u &= 0, \quad \text{ in } \Omega, \\
\gamma_{\C_0 }u|_{\p \Omega} &= \gamma_{\C_0 }u_s|_{\p \Omega},
\end{aligned}
\end{equation*}
since $ \nabla \cdot u_s = 0$. 
It is worth to note that the well-posedness of \eqref{eq_constNavier} is guaranteed 
when either a Neumann or  Dirichlet condition is specified. 
In the case of a Dirichlet condition this means specifying $u|_{\p \Omega}$.
In the case of the vector Helmholtz
equation  in \eqref{eq_vecHelmholtz} one only needs a Dirichlet type condition that specifies 
the tangential part of the vector field, i.e. $\nu \times u_s$, on the boundary.
The solutions to the vector Helmholtz equation form thus a proper subset of solutions to \eqref{eq_constNavier}
with zero divergence.

\subsection{Solutions with zero divergence and finite constraints} \label{sec_zero_div}

Here we prove the existence of divergence free solutions $u$ to
\begin{equation} \label{eq_curlBVP}
\begin{aligned}
\Delta u_s + \omega^2 \frac{\rho_0}{\mu_0} u_s &= 0, \quad \text{ in }\Omega, \\
 \nu \times u_s|_{\p \Omega} &= f, 
\end{aligned}
\end{equation}
by relating it to a second order system of Maxwell type.
And show that these solutions  also yield zero divergence 
solutions to the Navier equation, that can additionally be made to 
satisfy some given finite dimensional constraint of the same form as 
in Lemma \ref{lem_monotonicity_ineq1}.

In order to study the vector Helmholtz equation in \eqref{eq_curlBVP} we will introduce
some additional concepts. Firstly we define the function spaces
\begin{align*} 
\Hscurl{s}{\Omega}
:=
\{ u \in  H^s ( \Omega)^3 \,:\, \nabla \times  u \in  H^s ( \Omega)^3 \},
\end{align*}
$s \in \R$.
To specify how the tangential trace operator acts we need the spaces
\begin{align*} 
H_t^{s}(\p\Omega)  
:= \{ f \in  H^{s} (\partial \Omega)^3 \,:\,  
\nu \cdot f|_{\p\Omega} = 0
% \,\nabla_{\p} \cdot F \in H^{s} (\partial \Omega)^3 
\}, 
\end{align*}
$s \in \R$.
The tangential trace operator is determined by the mapping $u \mapsto \nu \times u |_{\p \Omega}$, and
maps
\begin{equation} \label{eq_tang_trace}
\begin{aligned}
(\nu \times \, \cdot)|_{\p \Omega} : \Hscurl{s}{\Omega} \to  H_t^{s-1/2}(\p\Omega), 
\end{aligned}
\end{equation}
continuously.
For more on this see \cite{Po22,Ce96} and the references therein. 
Lastly we use the notation\footnote{More explicitly
$
L^2(\Omega ; \operatorname{div}0) := \{ u \in L^2(\Omega)^3 \; : \; (u,\ \nabla \varphi)_ {L^2(\Omega)^3} = 0,\,
\forall\varphi \in H^1_0(\Omega)\}.
$}
$$
L^2(\Omega ; \operatorname{div}0) := 
\{ u \in L^2(\Omega)^3 \; : \; \nabla \cdot u = 0 \}.
$$
The following lemma gives the existence and uniqueness of zero divergence solutions to the vector Helmholtz
equation \eqref{eq_curlBVP2}, modulo a possible  eigenspace corresponding of zero.

\begin{prop} \label{prop_div0_sol}
There exists a weak solution $u_s \in H^1(\Omega)^3$ to the boundary value problem 
\begin{equation} \label{eq_curlBVP2}
\begin{aligned}
\Delta u_s + \omega^2 \frac{\rho_0}{\mu_0} u_s &= 0,  \quad \text{ in } \Omega, \\
(\nu \times  u_s)|_{\p\Omega} &= f,  \quad \text{on } \p \Omega, 
\end{aligned}
\end{equation}
with $\nabla \cdot u_s = 0$ in $\Omega$, for every $f \in  H^{1/2}_t(\p \Omega)$,
provided that
\begin{equation} \label{eq_f_cond_2}
\begin{aligned}
( f, \nabla \times \phi)_{L^2( \p \Omega)^3} = 0, \qquad \forall \phi \in \mathcal E,
\end{aligned}
\end{equation}
where $\mathcal E$ is as in Lemma \ref{lem_maxwell_bc3}.
The solution $u_s$ is moreover unique as an element in $H^1(\Omega)^3 /\mathcal E$. 
\end{prop}

\begin{proof}
We reduce the problem to solving a second order system of Maxwell type.
A solution $v \in H^1(\Omega)^3$ to the equations
\begin{equation} \label{eq_maxwell}
\begin{aligned}
-\nabla \times \nabla \times v + \omega^2 \frac{\rho_0}{\mu_0} v &= 0,  \quad \text{ in } \Omega, \\
(\nu \times  v)|_{\p\Omega} &= f,  \quad \text{on } \p \Omega, 
\end{aligned}
\end{equation}
also solves \eqref{eq_curlBVP2}, 
since $\nabla \cdot v = 0$, which follows by taking the divergence of the  first equation,
and since $\Delta v = - \nabla \times \nabla \times v + \nabla \nabla \cdot v= - \nabla \times \nabla \times v $.
We see similar that a solution to \eqref{eq_curlBVP2} is a solution to  \eqref{eq_maxwell}. 
The claim follows thus from Lemma \ref{lem_maxwell_bc3}.
\end{proof}

\noindent
We want to obtain solutions \eqref{eq_constNavier} with zero divergence. 
We need these solutions however to be such that they satisfy a finite dimensional constraint
of the type that appears in the monotonicity inequality of Lemma \ref{lem_monotonicity_ineq1},
namely that $\gamma_\C u \in V^\perp$, where $\dim(V) < \infty$.
We need thus to set an additional condition on the solutions of the vector Helmholtz
equation \eqref{eq_curlBVP2} to make them satisfy this constraint.
Our goal will be to find a set $W_1 \subset H^{1/2}_t(\p \Omega)$, so that 
solutions $u$ of  \eqref{eq_curlBVP2} 
with $ (\nu\times u)|_{\p \Omega} = f \in W_1$,  will satisfy  that $\gamma_\C u \in V^\perp$.

\medskip
\noindent
Let $V \subset L^2(\p\Omega)^3$ be a finite dimensional subspace. 
We define the map $\mathbf{S}$ as
\begin{equation} \label{eq_Sdef}
\begin{aligned}
&\mathbf{S}: f \mapsto \gamma_{\C_0} u_f, \qquad \mathbf{S}: X \to Y, \\ 
&X:= H_t^{1/2}(\p \Omega) \cap \big \{ (\nabla \times \phi )|_{\p \Omega} \,:\, \phi \in \mathcal E \big\}^\perp, 
\quad Y:= \mathbf{S} X
\end{aligned}
\end{equation}
where $u_f$ solves \eqref{eq_curlBVP2} and where we consider $Y$ as subspace in $L^2(\p \Omega)^3$.
One can readily  show that $X \subset H_t^{1/2}(\p \Omega)$ is a closed subspace.\footnote{
Note also that it easy to see that $(\nabla \times \phi )|_{\p \Omega} \in H^{1/2}_t(\p \Omega)$, 
since the  regularity of the eigenfunctions $\phi\in\mathcal E$  only depend on regularity of the boundary. Since 
$H^{1/2}_t(\p \Omega) = H^{1/2}_t(\p \Omega)\cap  (\nabla \times \mathcal{E})^\perp + \nabla \times \mathcal E$,
one sees that $\dim(X) = \infty$.
}

\begin{lem} \label{lem_Yclosed}
We have that $Y \subset L^2(\p\Omega)^3$ as defined in \eqref{eq_Sdef} is a closed subspace.
\end{lem}

\begin{proof}
Let $g_k \in Y$, $k\geq 1$ be a sequence such that $g_k \to g_0$ in $L^2(\p \Omega)^3$.
We want to show that $g_0 \in Y$.
Let $u_i \in H^1(\Omega)^3$, $i\geq0$ be the corresponding solutions to the boundary value problem
\begin{align*}
\begin{cases}
L_0 u_i &= 0, \\
\gamma_{\C_0} u_i &= g_i,
\end{cases}
\end{align*}
which exist, since we assume that $\omega$ is a non resonant frequency.
For $k\geq 1$ we have that 
\begin{align*}
\begin{cases}
L_0 (u_0- u_k) &= 0, \\
\gamma_{\C_0} (u_0- u_k) &= g_0-g_k.
\end{cases}
\end{align*}
Note that since $g_k \in Y$, we have by the unique solvability of \eqref{eq_curlBVP2}
that $u_k \equiv u_s + \phi$, where $u_s$ solves \eqref{eq_curlBVP2} and $\phi \in \mathcal E$.
Thus we know that
$$
\nabla \cdot u_k = \nabla \cdot (u_s + \phi) = 0, \quad \text{ for all } k\geq 1.
$$
By the well-posedness of the above boundary value problem we have that
$$
\| \nabla \cdot u_0 \|_{L^2(\Omega)^3} \leq  \| u_0 - u_k \|_{H^1(\Omega)^3} \leq \| g_0 - g_k \|_{L^2(\p \Omega)^3}
\to 0.
$$
Thus $ \nabla \cdot u_0 = 0$. 
Moreover $u_0 \in H^1(\Omega)^3$, and $\p \Omega$ is $C^{1,1}$,
and thus  $f := \nu \times u_0 \in H_t^{1/2}(\p \Omega)$. 
There exists $f_k \in X$, $k \geq 1$, so that $g_k = \mathbf{S}f_k$. Let $v_k$ be solutions
of \eqref{eq_curlBVP2} corresponding to the $f_k$, $k\geq 1$. Then since $\omega$ is a non-resonance frequency
and $g_k = \mathbf S f_k$, we have by uniqueness of solutions to \eqref{eq_bvp1}, that $u_k = v_{k}$.
Now since $f_k \in X$, we have by \eqref{eq_f_cond_2} that
$$
0 =
( f_k, \nabla \times \phi)_{L^2(\p \Omega)^3}  = 
(\nu \times v_k, \nabla \times \phi)_{L^2(\p \Omega)^3}  = 
(\nu \times u_k, \nabla \times \phi)_{L^2(\p \Omega)^3},\quad k \geq 1. 
$$
It follows that 
\begin{align*}
(\nu\times u_0, \nabla \times \phi)_{L^2(\p \Omega)^3}
&=
\big(\nu\times u_0 - \nu \times u_k , \nabla \times \phi\big)_{L^2(\p \Omega)^3}\\
&\leq C\| \nu \times  u_0 - \nu \times  u_k \|_{L^2(\p \Omega)^3} \\
&\leq  C\| u_0 - u_k \|_{H^1(\Omega)^3} \leq \| g_0 - g_k \|_{L^2(\p \Omega)^3}
\to 0.
\end{align*}
so that condition \eqref{eq_f_cond_2} holds for $f= \nu \times u_0$.
Hence we see that $u_0$ solves problem 
\eqref{eq_curlBVP2}, with $f = \nu \times u_0$.
Thus $g_0 = \gamma_{\C_0} u_0 = \mathbf S f$, and $g_0 \in Y$.
\end{proof}

\noindent
Since $Y \subset L^2(\p \Omega)^3$ is a closed subspace
by the previous lemma we can make the definition
$$
V_Y := \operatorname{Pr}_Y V, \qquad V'_Y := \{ y \in Y \,:\, y \perp V_Y   \},
$$
where $\operatorname{Pr}_Y$ is the orthogonal projection to $Y$ in $L^2(\p \Omega)^3$.
Note that $V_Y \subset Y$ is a closed subspace of Y, and that $V'_Y = V_Y^{\perp_Y}$, where $\perp_Y$ stand for the 
orthogonal complement in the subspace $L^2(Y)^3$. We thus have the orthogonal decomposition of $Y$ as
$$
Y = V_Y + V'_Y
$$
Furthermore we define the pre-images
\begin{equation} \label{eq_def_W0_W1}
\begin{aligned}
W_0 :=  \mathbf{S}^{-1} V_Y, \qquad W_1 :=  \mathbf{S}^{-1}V'_Y,
\end{aligned}
\end{equation}
Notice that
\begin{equation} \label{eq_dim_W0}
\begin{aligned}
\dim(W_0) < \infty.
\end{aligned}
\end{equation}
This is because
the mapping $ \mathbf{S}$ has at most a finite dimensional kernel, since if $u_s$ solves 
\eqref{eq_constNavier}, and $\gamma_{\C_0} u_s = 0$, then $u_s$ is in the eigenspace 
of zero for \eqref{eq_constNavier},
which is finite dimensional. We now have that $\dim(W_0) < \infty$, because of this and because
$W_0$ is the pre-image of $V_Y$, which has a finite dimension.
Later on we will need the following elementary lemma.

\begin{lem} \label{lem_X_split} We have that
\begin{equation} \label{eq_X_split}
\begin{aligned}
X =  W_0 + W_1.
\end{aligned}
\end{equation}
\end{lem}

\begin{proof}
To prove the claim suppose that $x \in X$, then $y:=\mathbf{S}x \in Y$. Decompose $y = y_0 + \tilde y$, where
$$
y_0 := \operatorname{Pr}_{V_Y} y \in V_Y,\qquad \tilde y := y - y_0 \in V'_Y.
$$
Now consider an element  $x' \in X$ such that
$$
x' = x_0 + \tilde x, \qquad x_0 \in \mathbf{S}^{-1}y_0, \quad \tilde x \in  \mathbf{S}^{-1}\tilde y
$$
We have that 
$$
\mathbf{S}(x-x') =  \mathbf{S}x - y_0 - \tilde y = 0.
$$
So that $z := x-x' \in \operatorname{Ker}(\mathbf{S}) \subset \mathbf{S}^{-1}V_Y$.
Now
$$
x = z + x' \in \mathbf{S}^{-1}V_Y + \mathbf{S}^{-1}V_Y + \mathbf{S}^{-1}V'_Y = W_0 +W_1.
$$
Thus \eqref{eq_X_split} holds. 
\end{proof}

\noindent 
The following lemma shows that the requirement $f \in W_1$ gives us a solution $u$ with a boundary value
that satisfies the desired property, i.e. $\gamma_\C u \perp V$.

\begin{lem} \label{lem_tildeV}
We have that 
\begin{equation} \label{eq_tildeV_impl}
\begin{aligned}
f \in W_1 \quad \Rightarrow \quad \mathbf{S} f \perp V.
\end{aligned}
\end{equation}
\end{lem}

\begin{proof}
We assume that $f \in W_1$ and hence $\mathbf{S}f \in V'_Y$ and therefore $\mathbf{S}f \perp V_Y$.
We now claim that $\mathbf{S}f \perp V$. To this end assume that $v \in V$, and split
$$
v = \operatorname{Pr}_Y v + \operatorname{Pr}_{Y^\perp}v.
$$
We then have that 
$$
(\mathbf{S}f\,,v)_{L^2(\p \Omega)^3} = 
(\mathbf{S}f\,, \operatorname{Pr}_Y v )_{L^2(\p \Omega)^3} + 
(\mathbf{S}f\,, \operatorname{Pr}_{Y^\perp}v)_{L^2(\p \Omega)^3} = 0,
$$
where the first term on the r.h.s. is zero because $\mathbf{S}f \perp V_Y$, 
the second term is zero, because  $\mathbf{S}f \in Y$. We thus see that 
$
\mathbf{S}f \perp v, \forall v \in V,
$
which is what we wanted to prove.
\end{proof}

\noindent
We summarize the above results in the following lemma.

\begin{lem} \label{lem_V_to_tildeV}
There exists a subspace  $ W_1 \subset X$  such that if $f \in W_1$ and  $u_s \in H^1(\Omega)^3$ 
is a solution to the Dirichlet problem
\begin{align} \label{eq_tildeL0}	
\begin{cases}
\tilde L_0 u_s &= 0, \quad \text{ in } \; \Omega\\
 \nu \times u_s|_{\p \Omega} &= f,
\end{cases}
\end{align}
with $ \nabla \cdot u_s = 0$, then $u_s$ coincides with a solution $u \in H^1(\Omega)^3$ to
\begin{align} \label{eq_L0}	
\begin{cases}
L_0 u &= 0, \quad \text{ in } \; \Omega\\
\gamma_\C u |_{\p \Omega} &= u_s|_{\p \Omega},
\end{cases}
\end{align}
where $ \nabla \cdot u = 0$ and $\gamma_\C u |_{\p \Omega} \in V^\perp$.
\end{lem}		

\begin{proof}
% \textcolor{blue}{(This is tedious to check. Proof should probably be shortened, since it is 
% in the non weak case is straight forward once Lemma \ref{lem_tildeV} is proved.)}
Note that Lemma	 \ref{lem_tildeV} shows that $\gamma_\C u |_{\p \Omega} \in V^\perp$ holds.
Moreover since $ \nabla \cdot u_s = 0$, we have formally that 
$$
\tilde L_0 u_s = \Delta u_s + \omega^2 \frac{\rho_0  }{ \mu_0 }u_s = 0
\quad \Leftrightarrow \quad 
L_{0} u_s = \mu_0 \Delta u_s + (\lambda_0 + \mu_0) \nabla \nabla \cdot u_s + \omega^2 \rho_0 u_s = 0 
$$
which together with uniqueness shows  that the claim holds for regular $u_s$ and $u$.
This also holds for weak solutions to \eqref{eq_tildeL0} and \eqref{eq_L0}, which
can be readily checked.
%\footnote{ \textcolor{blue}{The weak form of the proof is now in the comments of the latex file.}}
These arguments are standard and we omit the details.
\end{proof}

\subsection{Localization with Runge approximation for the constant coefficient operator} \label{sec_runge}
Here we  will study the localization of zero divergence solutions, that satisfy some finite dimensional 
constraint on the boundary, when the material parameters $\lambda,\mu$ and $\rho$ are constant.
More particularly we have the following result which is a counter part to Proposition \ref{prop_loc2}.

\begin{prop} \label{prop_loc_div0}
Assume that $D_1,D_2,D'_1,D'_2 \subset \Omega$ are as in Proposition \ref{prop_loc2}. 
Let $V\subset L^2(\p \Omega)^3$ be a subspace with $\dim (V)<\infty$,
then there exists a sequence of boundary conditions $g_j \in L^2(\p \Omega)^3$,
such that $ g_j \perp V$ in the $L^2(\p \Omega)^3$-norm,
$
\nabla \cdot u_j = 0,
$
and
$$
\| u_j \|_{ L^2(D_1)^3 } \to 0, 
\qquad \| \hat \nabla u_j \|_{ L^2(D'_1)^{3\times 3} } 
\to 0, 
$$
and for which
$$
\| u_j \|_{ L^2(D_2)^3 } \to  \infty,
\qquad \| \hat \nabla u_j \|_{ L^2(D'_2)^{3\times 3} } 
\to \infty, 
$$
as $j \to \infty$, and where $u_j$ solves \eqref{eq_bvp1}
with constant material parameters $\lambda_0, \mu_0$ and $\rho_0$.
\end{prop}

\begin{proof}
We set $D:= D_1 \cup D_2$.
Let $W_1 \subset H_t^{1/2}(\Gamma_N)$ be the subspace given by Lemma \ref{lem_V_to_tildeV}.
By this Lemma it is enough to construct the sequence of solutions $u_j$ to the vector Helmholtz equation,
$$
\Delta u_j + \omega^2\frac{\rho_0}{\mu_0} u_j = 0,
\qquad \nu \times u_j|_{\Gamma_N} = f \in W_1,
$$
% for a particular finite dimensional subspace $\tilde V \subset H^{1/2}_t(\p \Omega)$.
By the Proposition \ref{lem_blowupCand_div0} below,
we can find a $w \in \mathcal{S}_D$ with $ \nabla \cdot w = 0$ such that $w|_{D_1} = 0$ and 
$w|_{D_2}, \hat \nabla w|_{D_2} \neq 0$. By Lemma \ref{lem_approxCrit},
we can moreover approximate $w$  by a sequence of solutions $w_j \in H^1(D)^3$
in the $L^2(D)^3$-norm.

The desired sequence $u_j$ can now be defined
as $u_j:= j w_j$. One can now readily show that the $u_j$ obey the limits of the claim.
For details see the proofs of Propositions 5.2 and 5.1 in \cite{EP24}.
\end{proof}

\noindent
To complete the proof of Proposition \ref{prop_loc_div0} we need to construct the vector field $w$
and show that it can be approximated in the desired way.
To this end let $D = D_1 \cup D_2 \Subset \Omega$, where $D_1$ and $D_2$ are as in \eqref{eq_Dassump}.
It is enough to construct $w$  as a solution to the Maxwell type system
\begin{equation} \label{eq_maxwell_2}
\begin{aligned}
\hat L_0 w := -\nabla \times \nabla \times w+ \omega^2 \frac{\rho_0}{\mu_0} w &= 0,  \quad \text{ in } D, \\
(\nu \times  w)|_{\p D} &= f,  \quad \text{on } \p D,
\end{aligned}
\end{equation}
since solutions to \eqref{eq_maxwell_2}, give directly divergence free solutions to the vector Helmholtz equation,
by the argument in the proof of Proposition \ref{prop_div0_sol}.

We will use a modification of the Runge type arguments in \cite{EP24} and \cite{HPS19b} applied to equation \eqref{eq_maxwell_2}.
Define the set of solutions
$$
\mathcal{S}_\Omega := 
\{ u \in H^1(\Omega)^3 \,:\, 
\hat L_0 u = 0 \text{ in $\Omega$}, \; \nu \times u |_{\p \Omega} \in W_1 \},
$$
where $W_1$ is the subspace given by \eqref{eq_def_W0_W1}.
% Note that by Propositions \ref{prop_div0_sol} all $u \in \mathcal{S}_\Omega$ have 
Note that $\nabla \cdot u = 0$ for $u \in \mathcal{S}_\Omega$, and 
that by Lemma \ref{lem_V_to_tildeV} all $u \in \mathcal{S}_\Omega$ solve the constant coefficient Navier equation,
with $\gamma_{\C_0} u |_{\Gamma_N} \in V^\perp$.
We are interested in using the elements of $\mathcal{S}_\Omega$ to approximate solutions over a smaller set $D \Subset \Omega$ 
in the $L^2(D)^3$-norm, or more specifically solutions in the set
$$
\{ u \in H^1(D)^3 \,:\, \hat L_{0} u =  0 \text{ in } D \}. %, \text{ with } \nabla \cdot u = 0\}.
$$
The following lemma gives a criterion for which functions in the above set can be approximated
by functions in $\mathcal{S}_\Omega$.

\begin{lem} \label{lem_approxCrit}
Every $u \in \mathcal{S}_D$, where
$$
\mathcal{S}_D := \{ u \in H^1(D)^3 \,:\, \hat L_{0} u =  0 \text{ in } D,\, 
u \perp \mathcal{F} \},
$$
can be approximated by elements in $\mathcal{S}_\Omega$ in the $L^2(D)^3$-norm, where $\mathcal{F}$ is given by
\begin{align*}
\mathcal{F} := \spn\big\{ F \in L^2(D;\operatorname{div}0) \,:\, F_0 \perp \mathcal E, \, 
\nu \times \nu \times (\nabla \times SF)|_{\p \Omega} \in U  \big\},
\end{align*}
where $U := W^\perp_1 \cap X$ 
and $S: F \mapsto u_F$ is the source-to-solution operator of the problem \eqref{eq_bvp_w} below,
and $\mathcal E$ is as in Lemma \ref{lem_maxwell_source3}, and $F_0$ is the extension by zero of $F$ to $\Omega$.
\end{lem}

\begin{proof}
Denote the restrictions of elements of $\mathcal{S}_\Omega$ by
$$
\mathcal{S}_\Omega|_D := \{ u |_D \,:\, u \in \mathcal{S}_\Omega \}.
$$
Notice firstly that
$$
Z : =L^2(D;\operatorname{div}0) \subset L^2(D)^3,
$$
is a  closed subspace. This follows e.g. from the Hodge decomposition
$L^2(D)^3 = \operatorname{Ker}(\operatorname{div}) + \nabla H^1_0(\Omega)$, which is orthogonal.
See Theorem 10' p.54 in \cite{Ce96}.
It will be convenient to use the notation
$$
u \perp_Z v \qquad \Leftrightarrow \qquad u \in Z \text{ and } v \in Z \text{ are orthogonal in } L^2(D)^3.
$$
Now consider  $F \in (\mathcal{S}_\Omega|_D)^{\perp_Z}$, 
and let  $\varphi|_D \in \mathcal{S}_\Omega|_D$.
Define $w \in H^1(\Omega)^3$ as the weak solution to the source problem
\begin{align}  \label{eq_bvp_w}
\begin{cases}
\hat L_0 u_F &=  F_0 , \quad \text{ in } \quad \Omega, \\
\nu \times u_F|_{\p \Omega} &= 0.
\end{cases}
\end{align}
The existence of the solution $u_F$ follows from Lemma \ref{lem_maxwell_source3} and the assumption that $\omega$
is not a resonance frequency.
Now since $F \in (\mathcal{S}_\Omega|_D)^{\perp_Z}$, and since both $u_F$ and $\varphi$ are solutions on $\Omega$,
and since $\nu \times u_F|_{\p \Omega} = 0$,
we have that
\begin{align*} 
0 = (F,\varphi|_D)_{L^2(D)^3} 
&= - \int_{\Omega} \nabla \times u_F \cdot \nabla \times \varphi - \omega^2 \frac{\rho_0}{\mu_0} u_F \cdot \varphi\,dx 
- \int_{\p \Omega} \nu \times ( \nabla \times u_F) \cdot \varphi  \, dx  \\
&= \big ( \nu \times \nu \times ( \nabla \times SF), \, \nu \times \varphi \big)_{L^2(\p \Omega )^3},
\end{align*}
where $SF = u_F$, and $S$ is the source-to-solution operator
of the problem \eqref{eq_bvp_w} and where we used the vector decomposition $\varphi =(\varphi \cdot \nu)\nu - \nu \times \nu \times \varphi$
in the last step.
Using the definition of $\mathcal{S}_\Omega$ we rewrite this as
\begin{align*} 
\nu \times \nu \times ( \nabla \times SF) \perp \nu \times \varphi,\;\forall \varphi \in \mathcal{S}_\Omega
&\quad \Leftrightarrow \quad
\nu \times \nu \times ( \nabla \times SF) \in W_1^\perp \cap X = U. %\\
% &\quad \Leftrightarrow \quad
% \nu \times \nu \times ( \nabla \times SF) \in \tilde V^{\perp \perp} = \tilde V.
\end{align*}
Thus we have that 
\begin{small}
$$
F \in (\mathcal{S}_\Omega|_D)^{\perp_Z}  \; \Leftrightarrow \;   
F \in \spn \big\{ F \in L^2(D;\operatorname{div}0)^3 \,:\, \nu \times \nu \times ( \nabla \times SF) \in U \big\}
=\mathcal{F}.
$$
\end{small}
Hence for $u \in H^1(D)^3 \cap Z$ we have, because $\mathcal{S}_\Omega|_D \subset Z$ that
\begin{align*} 
u \in \operatorname{cl}_{L^2(D)^{3}}(\mathcal{S}_\Omega|_D) 
= u \in \operatorname{cl}_{Z}(\mathcal{S}_\Omega|_D) =(\mathcal{S}_\Omega|_D)^{\perp_Z \perp_Z} 
\;\Leftrightarrow\; u \perp_Z (\mathcal{S}_\Omega|_D)^{\perp_Z}
\;\Leftrightarrow\; u \perp \mathcal{F},
\end{align*}
where $\operatorname{cl}_{Y}$ stands for the closure corresponding to  the function space $Y$. This proves the claim.
\end{proof}

\noindent
We will need the following observation.

\begin{lem} \label{lem_U_fin}
We have that 
$$
\dim (U)  < \infty,
$$
where $U = \dim (W_1^\perp \cap X)$.
\end{lem}

\begin{proof}
Let $b \in W_1^\perp \cap X$.
By Lemma \ref{lem_X_split} we can write 
$
b = w_1 + w_0,
$
with $w_0 \in W_0$ and $w_1 \in W_1$.
Since $b \in W_1^\perp = \overline W_1^\perp$, we have that
$$
w_1 = - \operatorname{Pr}_{\overline{W}_1} w_0,  \qquad\text{ and } \qquad 
b =  w_0 - \operatorname{Pr}_{\overline{W}_1} w_0.
$$
Since $\dim(W_0) = M < \infty$, we have a basis $\{ \phi_k \}_{k=1}^M$ of $W_0$.
Let $w_0 = \sum_{k=1}^M c_k \phi_k$. It follows that
$$
b =   w_0 - \sum_{k=1}^M c_k \operatorname{Pr}_{\overline{W}_1} \phi_k.
$$
The set 
$$
\{ \phi_1, \dots ,\phi_M , \operatorname{Pr}_{\overline{W}_1} \phi_1, \dots ,  \operatorname{Pr}_{\overline{W}_1} \phi_M \}
$$
is thus a basis of $W_1^\perp \cap X$, and $\dim(W_1^\perp \cap X) \leq 2M$.
\end{proof}

\noindent
We now prove the existence of a solution $w$ in the set $\mathcal{S}_D$ with zero divergence,
that we need in the proof of the  localization result in Proposition  \ref{prop_loc_div0}.
% We obtain such as solution from the following lemma.

\begin{lem} \label{lem_blowupCand_div0}
Let $D_1$ and $D_2$ be as in \eqref{eq_Dassump} and 
$\mathcal{S}_D$ as in Lemma \ref{lem_approxCrit}. There exists a solution 
$w \in \mathcal{S}_D$, with $ \nabla \cdot w = 0$ and  $w=0$ in $D_1$, and 
$$
\| w \|_{ L^2(D_2)^3 } \neq 0, \qquad 
\| \hat \nabla w \|_{ L^2(D_2)^{3\times 3} } \neq 0.
% \;\| \nabla \cdot w \|_{ L^2(D_2)} \neq 0.
$$
\end{lem}

\begin{proof} 
Let $D = D_1 \cup D_2$. We will show that we can pick the required $w \neq 0$ as
$$
\hat L_0 w_{ f} = 0,\quad \text{ in } D,
$$
where we set $w_{ f} \equiv  0$, in $D_1$, so that trivially
\begin{align}  \label{eq_wBVP0}
\hat L_0 w_{ f} = 0 \quad \text{ in } D_1,\qquad 
\nu \times  w_{ f} = 0\quad  \text { on } \p D_1,
\end{align}
and in $D_2$ we let $w_{ f}$ solve the boundary value problem
\begin{align}  \label{eq_wBVP}
 \hat L_0 w_{ f} = 0 \quad \text{ in } D_2,\qquad 
 \nu \times  w_{ f} = { f}\quad  \text { on } \p D_2.
\end{align}
We derive two conditions on the boundary value ${ f} \in H^{1/2}_t(\p D_2)$,
which will guarantee that
the desired properties hold and then show that these conditions can be satisfied.

\medskip
\noindent
\textit{Condition 1.} 
It might happen that zero is an eigenvalue of the eigenvalue problem corresponding to
\eqref{eq_wBVP}, we thus require that 
\begin{align}  \label{eq_condN}
{ f} \perp  \nabla \times \mathcal{E}_{D}
\end{align}
where $\mathcal{E}_D$ is the finite dimensional subspace corresponding to a possible zero eigenvalue,
given by Lemma \ref{lem_maxwell_bc3}.
Note that this guarantees that we have a unique solution, when restricting the boundary data suitably, 
and that we can define the boundary data map as 
$$
\hat \Lambda:  f \mapsto \nu \times (\nabla \times w_{ f})|_{\p D_2} , \qquad
\hat \Lambda: ( \nabla \times \mathcal{E}_D)^\perp \cap H_t^{1/2}(\p D_2) \to H_t^{1/2}(\p D_2).
$$

\medskip
\noindent
\textit{Condition 2.} 
The second condition on ${ f}$ will guarantee that $w_{ f} \in \mathcal{S}_D$.
In order to obtain a $w_{ f} \in \mathcal{S}_D$, we know by  the definition of $\mathcal{S}_D$ in
Lemma \ref{lem_approxCrit} that it is enough
to pick a boundary condition ${ f}$ such that 
\begin{equation} \label{eq_wf_ortho}
\begin{aligned}
w_{ f} \perp \mathcal{F}.
\end{aligned}
\end{equation}
The first step is to rewrite this as a condition on the boundary value ${ f}$. 
Suppose that this condition holds for $w_{ f}$ solving \eqref{eq_wBVP0}--\eqref{eq_wBVP}.
Notice that we can choose $w_f \perp \mathcal{E}_D$, since $w_f$ is unique modulo elements in $\mathcal{E}_D$.
It follows that the above condition is
$$
0 = (w_f, F)_{L^2(D_2)^3} \quad \Leftrightarrow \quad 0 = (w_f, \operatorname{Pr}_{\mathcal{E}_D^\perp}F)_{L^2(D_2)^3}.
$$
We can thus assume that $F \perp \mathcal{E}_D$. Now if \eqref{eq_wf_ortho} holds, then
for any $F \in \mathcal{F}$ and corresponding $u_F$ that solves \eqref{eq_bvp_w}, we have that
\begin{align*} 
0 = (w_{ f}, F)_{L^2(D)^3} = \int_{D_2} \nabla \times w_{ f} \cdot \nabla \times u_F 
- \omega^2 \frac{\rho_0}{\mu_0} w_{ f} \cdot u_F \, dx 
+ \int_{\p D_2} w_{ f}\cdot \nu \times (\nabla \times u_F ) \,dS
\end{align*}
where we used the fact that $u_F$ restricted to $D_2$ is a solution in $D_2$ to the boundary value problem 
with the source term $F$.
Since $w_{ f}$ solves \eqref{eq_wBVP} on $D_2$, the above equation implies that
\begin{equation} \label{eq_tildeLambda_eq}
\begin{aligned}
0 =  \big(   f  ,\, \nu \times (\nabla \times u_F) \big)_{L^2(\p D_2)^3} 
- \big( \nu \times \nu \times u_F ,\, \hat \Lambda  f \big)_{L^2(\p D_2)^3}.
\end{aligned}
\end{equation}
The boundary data map $\hat \Lambda$ is 
symmetric
% \footnote{ \textcolor{blue}{Note however that $\hat \Lambda$ is not bounded on $L^2$ and
% it is thus not a bounded self-adjoint operator on $L^2$.}} 
in the sense that
$$
\big( \hat \Lambda \phi ,\, \nu \times \psi \big)_{L^2(\p D_2)^3}
= \big( \nu \times \phi ,\, \hat \Lambda \psi \big)_{L^2(\p D_2)^3}, 
\qquad \forall \psi,\phi \in \mathcal{E}_D^\perp \cap H^{1/2}_t(\p D_2).
$$
We want to apply this to the above equation.
%and hence we need to check that  $\nu \times u_F |_{\p D_2} \in \mathcal{E}_D^\perp$.
Hence we need to check that  $\hat \Lambda(\nu \times u_F)$ is defined. 
This follows if we can show that \eqref{eq_wBVP}, has a solution when $f = \nu \times u_F$.
To this end note that 
\begin{align*}  %\label{eq_wBVP}
 \hat L_0 v_F = F  \quad \text{ in } D_2,\qquad 
 \nu \times v_F = 0 \quad  \text { on } \p D_2,
\end{align*}
has a solution by Lemma \ref{lem_maxwell_source3}, since $F \perp \mathcal{E}_D$
But now $v = u_F - v_F$ 
\begin{align*}  %\label{eq_wBVP}
 \hat L_0 v = 0  \quad \text{ in } D_2,\qquad 
 \nu \times v =  \nu \times u_F \quad  \text { on } \p D_2.
\end{align*}
So that $\hat \Lambda (\nu \times u_F) = \nu \times (\nabla \times v)$.
Thus we can use the above symmetry property and write \eqref{eq_tildeLambda_eq} in terms of an operator $T$ as 
$$
0 =  \big(  f, \,  \nu \times (\nabla \times u_F )  - \nu \times \hat \Lambda ( \nu \times u_F) \big)_{L^2(\p D_2)^3}
=: \big(  f ,\, T u_F  \big)_{L^2(\p D_2)^3}.
$$
We have moreover that $u_F = SF$, where $S$  is the source-to-solution operator of \eqref{eq_bvp_w}. 
It follows that ${ f}$ should satisfy the constraint 
\begin{align}  \label{eq_condW}
 f \perp \mathcal{W} \,:=\, \spn \{\, T(SF)  \,:\, F \in \mathcal{F} \},
\end{align}
We now show that 
$
\dim{\mathcal{W}}< \infty.
$
It is enough to show that
\begin{align}  \label{eq_Sv_dim}
N:= \dim \big(\spn \{ SF |_{\Omega\setminus D_2} \,:\, F \in \mathcal{F}\} \big) < \infty,
\end{align}
since $\mathcal{W}$ can be obtained from this by a linear map.
Note firstly that because $F \in \mathcal{F}$, we have  that
\begin{align*} 
\nu \times \nu \times ( \nabla \times SF )|_{\p \Omega}  \in U,
\end{align*}
so that the trace  $\nu \times \nu \times ( \nabla \times  \cdot )|_{\p \Omega}$ here maps into $U$.
By Lemma \ref{lem_U_fin} we know that $\dim(U) =  2M < \infty$.
Suppose that $N > 2M$. Then we have an $SF|_{\Omega\setminus D_2} \neq 0$, 
such that 
$$
 \nu \times \nu \times ( \nabla \times SF )|_{\p \Omega}= 0
\quad \Rightarrow \quad
 \nu \times ( \nabla \times SF )|_{\p \Omega}= 0,
$$
since the trace map is linear. Moreover $SF$ has a second zero boundary condition, since by \eqref{eq_bvp_w},
we have that  $\nu \times SF|_{\Gamma_N} = 0$. Thus both
$$
\nu \times SF|_{\Gamma_N} = 0, \qquad
\nu \times ( \nabla \times SF )|_{\Gamma_N}= 0.
$$
The unique continuation principle (see for instance Theorem 2.2 in \cite{HLL18}) gives now that $SF \equiv 0$,
in $\Omega \setminus D_2$ which is a contradiction. Thus  $\dim{\mathcal{W}}< \infty$.

\medskip
\noindent
The requirement on $ f \in H_t^{1/2}(\p D_2)$  is that
$$
f \;\perp\;  \nabla \times \mathcal{E}_D + \mathcal{W}.
$$
This is a finite dimensional constraint and can thus be satisfied by some $ f \neq 0$.
We see that $w_{ f}$ satisfies the desired properties, since  $w_{ f} \equiv 0$ in $D_1$, and 
by Proposition \ref{prop_div0_sol} we have that
$ \nabla \cdot w_{  f} |_{D_2} = 0$ and $ \hat \nabla w_{f} |_{D_2} \neq 0$.
\end{proof}

\section{Linearized monotonicity test} \label{sec_mono_shape}

\noindent
In this section we will derive the main Theorem underlying our linearized monotonicity tests.
We will consider inhomogeneities in the material parameters of the following form.
We let $D_1, D_2, D_3 \Subset \Omega$, and assume that $\lambda,\mu, \rho \in L_+^\infty(\Omega)$ 
are such that
\begin{equation} \label{eq_lambdaMuRho}
\begin{aligned} 
\lambda &= \lambda_0 + \chi_{D_1} \psi_\lambda , \qquad \psi_\lambda \in L^\infty(\Omega),
\quad \psi_\lambda > \delta, \\
\mu &= \mu_0 + \chi_{D_2} \psi_\mu, \qquad \psi_\mu \in L^\infty(\Omega),\quad \psi_\mu > \delta, \\
\rho &= \rho_0 - \chi_{D_3} \psi_\rho, \qquad \psi_\rho \in L^\infty(\Omega),\quad \rho_0 - \delta > \psi_\rho > \delta, 
\end{aligned}
\end{equation}
where the constants $\lambda_0,\mu_0,\rho_0 > 0$ and $\delta > 0$.
The coefficients $\lambda,\mu$ and $\rho$  model inhomogeneities in an otherwise homogeneous background medium
determined by $\lambda_0,\mu_0$ and $\rho_0$.

\noindent
The linearized monotonicity method is based on the following theorem.

\begin{thm} \label{thm_Lin_inclusionDetection}
Let $D := D_1 \cup D_2 \cup D_3$ where the sets are as in \eqref{eq_lambdaMuRho} and
$B \subset \Omega$  and $\alpha_j > 0$, 
and set $\alpha:=(\alpha_1,\alpha_2,\alpha_3)$. Let
$$
\mathcal{N} := \# \big\{\sigma \in \spec(\Lambda_0 + \Lambda'_0[\alpha_1,\alpha_2,-\alpha_3] - \Lambda)
\;:\; \sigma < 0 \big\},
$$
where $\Lambda_0$ and $\Lambda$ are the NtD-maps for the coefficients $\lambda_0,\mu_0,\rho_0$ and $\lambda,\mu,\rho$ respectively,
and where $\Lambda'_0[\alpha_1,\alpha_2,-\alpha_3] := \Lambda'_0[\alpha_1\chi_B,\alpha_2\chi_B,-\alpha_3\chi_B]$. 
There exists a $\gamma_0 > 0$ such that the following holds: \\
\begin{enumerate}

\item 
Assume that  $B \subset D_j$, for $j \in I$, for some $I \subset\{1,2,3\}$.  
Then for all $\alpha_j$ with $\alpha_j \leq \gamma_0$, $j \in I$, and $\alpha_j = 0$, 
$j \notin I$, we have that  $\mathcal{N} < \infty$.
\\

\item
If $B \not \subset \osupp(D)$, then for all $\alpha$, $|\alpha| \neq 0$, $\mathcal{N} = \infty$.\\

\item 
Assume $\rho = \rho_0$, and that $\Gamma_N = \p \Omega$.
If $B \not \subset \osupp(\mu)$, and we choose
$\alpha = (0,\alpha_2,0) \neq 0$, then  $\mathcal{N} = \infty$.
\\

\end{enumerate}
\end{thm}

\begin{proof}
We start by proving part $(1)$ of the claim.
It will be convenient to use the abbreviations 
$$
\bar \lambda := \lambda - \lambda_0
\qquad \bar \mu := \mu - \mu_0,
\qquad \bar \rho := \rho_0 - \rho,
$$
Now consider $w := u - u_0$. We see similarly as in Lemma \ref{lem_ab_1} that $w$ is a weak solution of 
the form \eqref{eq_weak3} of 
\begin{align*}  %\label{eq_bvp2}
\begin{cases}
\nabla \cdot (\C_0\,  \hat \nabla w )  + \omega^2\rho_0 w  &=  
\nabla \cdot (\C_{\bar \lambda, \bar \mu}\,  \hat \nabla u )  + \omega^2 \bar \rho u, \\
\;\quad\quad\quad\quad\quad(\gamma_{\C_0} w ) |_{\Gamma_N} &= ( \gamma_{ \C_{\bar \lambda, \bar \mu}} u ) |_{\Gamma_N}, \\	
 \quad\quad\quad\quad\quad\quad \quad w |_{\Gamma_D} &= 0,	
\end{cases}
\end{align*}
where $\C_0$ corresponds to $\lambda_0$ and $\mu_0$.
From the estimate in Proposition \ref{prop_elipEst_1}  with $A = \C_{\bar \lambda, \bar \mu}\,  \hat \nabla u$ 
and $F = \omega^2 \bar \rho u$, we easily see that
\begin{small}
\begin{align*}
\| w \|_{ H^1(\Omega)^3} 
\leq 
C \big( 
\| \lambda - \lambda_0\|_{L^\infty }\| \nabla \cdot u \|_{ L^2(D_1) }
+ \| \mu - \mu_0 \|_{L^\infty }\| \hat \nabla u\|_{ L^2(D_2)^{3\times 3}}
+ \| \rho_0- \rho \|_{L^\infty }\| u \|_{ L^2(D_3)^3 }
\big).
\end{align*}
\end{small}

% \medskip
% NEW END 
%
Next using this and the triangle inequality we have that
\begin{small}
\begin{align*} 
\int_{D_2}  2\bar \mu |\hat \nabla u_0 |^2\,dx 
\leq
 2 \int_{D_2}  2 \bar \mu (|\hat \nabla w |^2 +  |\hat \nabla u |^2)\,dx 
\leq C
\big( 
&\| \bar \lambda \|_{L^\infty }\| \nabla \cdot u \|^2_{ L^2(D_1) } 
+ \| \bar \mu \|_{L^\infty } \| \hat \nabla u\|^2_{ L^2(D_2)} \\
&+\| \bar \rho \|_{L^\infty }\| u \|^2_{ L^2(D_3)^3 }
\big).
\end{align*}
\end{small}
We obtain similarly, that
\begin{align*} 
\int_{D_1}  \bar \lambda |\nabla \cdot u_0 |^2\,dx 
\leq 
C \big( 
\| \bar \lambda \|_{L^\infty }\| \nabla \cdot u \|^2_{ L^2(D_1) } 
+\| \bar \mu \|_{L^\infty } \| \hat \nabla u\|^2_{ L^2(D_2)^{3\times 3}} 
+\| \bar \rho \|_{L^\infty }\| u \|^2_{ L^2(D_3)^3 }
\big),
\end{align*}
and that
\begin{align*} 
\int_{D_3} \bar \rho | u_0 |^2\,dx 
\leq C 
\big( 
\| \bar \lambda \|_{L^\infty }\| \nabla \cdot u \|^2_{ L^2(D_1) } 
+\| \bar \mu \|_{L^\infty } \| \hat \nabla u\|^2_{ L^2(D_2)^{3\times 3}} 
+\| \bar \rho \|_{L^\infty }\| u \|^2_{ L^2(D_3)^3 }
\big).
\end{align*}
We define the $\gamma_0' > 0$ as
$$
\gamma_0' := \min \Big \{ 
\min_{x \in D_2 } \frac{\bar \mu(x)}{\|\bar \mu \|_{L^\infty(D_2)}},\; 
\min_{x \in D_1 } \frac{\bar \lambda(x) }{\| \bar \lambda \|_{L^\infty(D_1) }},\;
\min_{x \in D_3 }  \frac{\bar \rho(x) }{\| \bar  \rho \|_{L^\infty(D_3) }}
\Big \} > 0,
$$
where the estimate holds because of \eqref{eq_lambdaMuRho}.
Using the inequality of Lemma \ref{lem_monotonicity_ineq1} and the three previous inequalities, 
and that  
$$
 \lambda - \lambda_0,\;  \mu - \mu_0,\;  \rho_0 - \rho \geq 0 \quad \text{ in }B,
$$
we obtain 
\begin{align}
\label{eq_mono_applied}
\big(  (\Lambda_0 - \Lambda)g, \,g  \big )_{L^2(\Gamma_N)^3} 
&\geq
\int_\Omega  2(\mu-\mu_0 ) |\hat \nabla u |^2 + (\lambda - \lambda_0 ) |\nabla \cdot u|^2 + \omega^2(\rho_0-\rho) |u|^2 \,dx \nonumber \\
&\geq 
\gamma_0' \big( 
\| \bar \lambda \|_{L^\infty }\| \nabla \cdot u \|^2_{ L^2(D_1) } 
+ \| \bar \mu  \|_{L^\infty }\| \hat \nabla u\|^2_{ L^2(D_2)^{3\times 3}}
+ \| \bar \rho \|_{L^\infty }\| u \|^2_{ L^2(D_3)^3 } 
\big) \\
&\geq
C \int_\Omega  2(\mu-\mu_0 ) |\hat \nabla u_0 |^2 + (\lambda - \lambda_0 ) |\nabla \cdot u_0|^2  + \omega^2(\rho_0-\rho) |u_0|^2 \,dx, 
\nonumber
\end{align}
with  some $C>0$ and $g \in V^\perp$.
From this and \eqref{bilinear_Frechet}  we get that
\begin{align*}  %\label{eq_mono1}
\big(  (\Lambda_0  + \Lambda_0'[\alpha_1,\alpha_2,-\alpha_3] - \Lambda)g, \,g  \big )_{L^2(\Gamma_N)^3} 
&\geq
C \int_\Omega  2(\mu-\mu_0 ) |\hat \nabla u_0 |^2 + (\lambda - \lambda_0 ) |\nabla \cdot u_0|^2  \,dx \\
&\quad+ \int_\Omega \omega^2(\rho_0-\rho) |u_0|^2 \,dx \\
&\quad - \int_\Omega  \alpha_2 \chi_B |\hat \nabla u_0 |^2 + \alpha_1 \chi_B |\nabla \cdot u_0|^2  - \alpha_3 \chi_B |u_0|^2 \,dx. 
\end{align*}
It follows that
\begin{align}  \label{eq_part1_pos}
\big(  (\Lambda_0  + \Lambda_0'[\alpha_1,\alpha_2,-\alpha_3] - \Lambda)g, \,g  \big )_{L^2(\Gamma_N)^3}  \geq 0, \quad g \in V^\perp,
\end{align}
provided that we set $\alpha_j = 0$, when $B \not \subset D_j$,
and that 
\begin{align*}  %\label{eq_mono1}
\alpha_1 \leq C\min_{D_1} (\lambda - \lambda_0 ), \quad
\alpha_2 \leq 2C\min_{D_2} (\mu-\mu_0 ),\quad
\alpha_3 \leq C\omega^2\min_{D_3}(\rho_0-\rho),
\end{align*}
if $B \subset D_j$.
We set $\gamma_0$ to be the minimum non-zero value of the right hand sides of the above minimums.
% The first part of the claim follows now from 
The inequality of \eqref{eq_part1_pos} implies by Lemma 6.1 in \cite{EP24} that $\mathcal{N} < \infty$.
This proves the first part of the claim.

\medskip
\noindent
Next we prove part (2) of the claim. 
We can assume that $B \cap D = \emptyset$, by considering a subset of $B$ if needed.
Assume that the claim is false and that $\mathcal{N} < \infty$, so that 
$(\Lambda_0 + \Lambda'_0[\alpha_1,\alpha_2,-\alpha_3] - \Lambda)$
has finitely many negative eigenvalues.
By Lemma 6.1 in \cite{EP24} we have a finite dimensional subspace $V_1 \subset L^2(\Gamma_N)$, such that  
\begin{equation} \label{eq_pos}
\begin{aligned}
\big((\Lambda_0 + \Lambda'_0[\alpha_1,\alpha_2,-\alpha_3] - \Lambda) g, \, g\big)_{L^2(\Gamma_N)^3}
\geq  0,  \qquad g \in V_1^\perp.
\end{aligned}
\end{equation}
We will obtain a contradiction from this.
From Lemma \ref{lem_monotonicity_ineq1} we get that
\begin{align} \label{eq_upper_part2}
\big( (\Lambda_0 + \Lambda'_0[\alpha_1,\alpha_2,-\alpha_3] - \Lambda) g, \,g  \big)_{L^2(\Gamma_N)^3} 
&\leq
\int_\Omega (\mu - \mu_0- \alpha_2 \chi_B ) | \hat \nabla u_0|^2 \,dx \nonumber \\ 
&\quad +\int_\Omega (\lambda - \lambda_0 - \alpha_1 \chi_B ) | \nabla \cdot u_0|^2 \,dx \\ 
&\quad +\int_\Omega \omega^2(\rho_0 - \rho - \alpha_3 \chi_B ) |  u_0|^2 \,dx \nonumber 
% &\quad - \int_\Omega   2 \alpha_1 \chi_B |\hat \nabla u_0 |^2 + \alpha_2 \chi_B |\nabla \cdot u_0|^2  \\ 
% &\quad \qquad + \omega^2 \alpha_3 \chi_B |u_0|^2 \,dx. 
\end{align}
where $u_0$ solves \eqref{eq_bvp1} with coefficients given by  $\lambda_0, \mu_0$ and $\rho_0$, and the  
boundary condition $g \in V_2^\perp$, where $V_2$ is a finite dimensional subspace. 
The  terms on right  hand side of \eqref{eq_upper_part2} can be split as 
\begin{align} \label{eq_mu_term} 
\int_\Omega (\mu - \mu_0- \alpha_2 \chi_B ) | \hat \nabla  u_0|^2 \,dx  
=
\int_{D_2} (\mu - \mu_0) | \hat \nabla  u_0|^2 \,dx  
- \int_{B} \alpha_2 \chi_B | \hat \nabla  u_0|^2 \,dx 
\end{align}
and 
\begin{align} \label{eq_lambda_term} 
\int_\Omega (\lambda - \lambda_0- \alpha_1 \chi_B ) | \nabla  \cdot u_0|^2 \,dx  
=
\int_{D_1} (\lambda - \lambda_0) |  \nabla \cdot u_0|^2 \,dx  
- \int_{B} \alpha_1 \chi_B | \hat \nabla  u_0|^2 \,dx 
\end{align}
and 
\begin{align} \label{eq_rho_term} 
\int_\Omega \omega^2(\rho_0 - \rho - \alpha_3 \chi_B ) |  u_0|^2 \,dx 
=
\int_{D_3} \omega^2(\rho_0 - \rho ) |  u_0|^2 \,dx 
- \int_{B} \omega^2\alpha_3 \chi_B  |  u_0|^2 \,dx 
\end{align}
We now use localized solutions that become large in the set $B$ and small in $D$.
By choosing the sets $D'_1$ and $D'_2$ in Proposition \ref{prop_loc2}, as  $D'_1 = \osupp(D)$ and $D'_2 = B$,
and suitable sets $D_1$ and $D_2$,
we get a sequence $g_j  = (\gamma_\C u_{0,j})|_{\p\Omega} \in (V_1 \cup V_2 )^\perp$ of
boundary data that give the  solutions $u_{0,j}$ to \eqref{eq_bvp1}, with the coefficients
$\lambda_0, \mu_0$ and $\rho_0$, such that
$$
\| u_{0,j} \|_{ L^2(\osupp(D))^3 }, 
\;\| \hat \nabla u_{0,j} \|_{ L^2(\osupp(D))^{3\times 3} }, 
\;\| \nabla \cdot u_{0,j} \|_{ L^2(\osupp(D))} 
\to 0, 
$$
and such that
$$
\| u_{0,j} \|_{ L^2(B)^3 }, 
\;\| \hat \nabla u_{0,j} \|_{ L^2(B)^{3\times 3} }, 
\;\| \nabla \cdot u_{0,j} \|_{ L^2(B)} 
\to \infty, 
$$
as $j \to \infty$.
If $|\alpha| \neq 0$, then it  follows from these limits and equations \eqref{eq_mu_term}, 
\eqref{eq_lambda_term} and \eqref{eq_rho_term}, that \eqref{eq_upper_part2} gives the estimate 
\begin{align*} 
\big( (\Lambda_0 + \Lambda'_0[\alpha_1,\alpha_2,\alpha_3] - \Lambda) g_j, \,g_j  \big)_{L^2(\Gamma_N)^3} < 0, 
\end{align*}
when $j$ is large enough, with $g_j \in (V_1 \cup V_2 )^\perp \subset V_1 ^\perp$. 
But this is in contradiction with \eqref{eq_pos}. Part $(2)$ of the claim thus holds.

\medskip
\noindent
Finally we prove part (3) of the claim. 
We can assume that $B \cap \osupp(\mu - \mu_0) = \emptyset$, by considering a subset of $B$ if needed.
We again assume that the claim is false and that there is a finite dimensional subspace $V_1 \subset L^2(\Gamma_N)^3$
such that \eqref{eq_pos} holds. Consider $u_0$, with $ \nabla \cdot u_0 = 0$.
By using Lemma \ref{lem_monotonicity_ineq1} we have that 
\begin{align} \label{eq_upper_part3}
\big( (\Lambda_0 + \Lambda'_0[ 0,\alpha_2,0] - \Lambda) g, \,g  \big)_{L^2(\Gamma_N)^3} 
&\leq
\int_\Omega (\mu - \mu_0- \alpha_2 \chi_B ) | \hat \nabla u_0|^2 \,dx, % \nonumber \\ 
% &\quad +\int_\Omega \omega^2(\rho - \rho_0 + \alpha_3 \chi_B ) |  u_0|^2 \,dx 
\end{align}
where $u_0$ solves \eqref{eq_bvp1} with coefficients given by  $\lambda_0, \mu_0$ and $\rho_0$, and the  
boundary condition $g \in V_2^\perp$, where $V_2$ is a finite dimensional subspace. The $\mu$ 
term on right  hand side  of \eqref{eq_upper_part3} can be written as   
\begin{align} \label{eq_mu_term_3} 
\int_\Omega (\mu - \mu_0- \alpha_2 \chi_B ) | \hat \nabla  u_0|^2 \,dx  
=
\int_{D_2} (\mu - \mu_0) | \hat \nabla  u_0|^2 \,dx  
- \int_{B} \alpha_2 \chi_B | \hat \nabla  u_0|^2 \,dx. 
\end{align}
We now use the localized solutions $u_{0,j}$ given by Proposition \ref{prop_loc_div0}.
By choosing the sets $D'_1$ and $D'_2$ in Proposition \ref{prop_loc_div0}, as  $D'_1 = \osupp(\mu - \mu_0)$ and $D'_2 = B$,
and suitable sets $D_1$ and $D_2$,
we get a sequence of boundary data $g_j \in V_2^\perp$, where $V_2$ is a finite dimensional subspace,
for which $ \nabla \cdot u_{0,j} = 0$, and
$$
% \| u_{0,j} \|_{ L^2( \osupp(\mu -\mu_0) )^3 } \to 0, 
% \qquad 
\| \hat \nabla u_{0,j} \|_{ L^2( \osupp(\mu -\mu_0))^{3\times 3} } 
\to 0, 
$$
as $j \to \infty$, and  
$$
% \| u_{0,j} \|_{ L^2(B)^3 } \to  \infty,
% \qquad 
\| \hat \nabla u_{0,j} \|_{ L^2(B)^{3\times 3} } 
\to \infty, 
$$
as $j \to \infty$, 
where $u_{0,j}$ solves \eqref{eq_bvp1}, with the boundary conditions $g_j$ and the coefficients 
$\lambda_0, \mu_0$ and $\rho_0$.
% For $\alpha_1 \neq$ or $\alpha_3 \neq 0$ we get from these limits and equations \eqref{eq_mu_term_3} and \eqref{eq_rho_term_3}, with
For $\alpha_2 \neq 0$   we get from these limits and equation \eqref{eq_mu_term_3}, with
\eqref{eq_upper_part3} that
\begin{align*} 
\big( (\Lambda_0 + \Lambda'_0[0,\alpha_2,0] - \Lambda) g_j, \,g_j  \big)_{L^2(\Gamma_N)^3} 
< 0, 
\end{align*}
when $j$ is large enough. 
But this is in contradiction with \eqref{eq_pos}, and we have proved part (3) of the claim.
\end{proof}

\subsection{Reconstructing  $\osupp(D)$ from $\Lambda$} \label{sec_reconstruct_D}
We can use parts (1) and (2) of Theorem \ref{thm_Lin_inclusionDetection} to reconstruct
$\osupp(D)$.
The linearized monotonicity method that we present here for recovering the set
$\osupp(D)$, is similar to the reconstruction for procedure in \cite{EP24} for the full Neumann-to-Dirichlet 
maps. In this subsection we review how to use Theorem \ref{thm_Lin_inclusionDetection} for the reconstruction
of $\osupp(D)$.
Recall that $D := D_1 \cup D_2 \cup D_3$ where the sets are as in \eqref{eq_lambdaMuRho}.
% The algorithm to reconstruct $\osupp(D)$, is described in \cite{EP24}.

\medskip
\noindent
To determine $\osupp(D)$ one uses conditions $(1)$ and $(2)$ of Theorem \ref{thm_Lin_inclusionDetection}.
Algorithm \ref{alg_shapeInclusion} gives the general outline of the reconstruction procedure
for $\osupp(D)$.
% Algorithm \ref{alg_shapeInclusion}
The aim is to generate a collection $\mathcal A$ of subsets of $\Omega$, 
such that $\cup \mathcal{A}$ approximates the set $\osupp(D)$. This 
gives an approximation of the shape of the inhomogeneous region $D$, disregarding any internal cavities.
Algorithm \ref{alg_shapeInclusion} works roughly by choosing a collection of subsets
$\mathcal{B} = \{ B \subset \Omega\}$, and building an approximating collection $\mathcal{A}$ 
by choosing $B$, such that $B \subset \osupp(D)$. An approximation  of $\osupp(D)$ is then obtained
as $\cup \mathcal{A}$, in which we  fill any internal cavities, where the algorithm might have given an indeterminate result.
Note that in Algorithm \ref{alg_shapeInclusion} $\Lambda$ denotes 
the measured Neumann-to-Dirichlet map, and $\Lambda_0$ the  Neumann-to-Dirichlet map
of the constant background, and its $\Lambda'_0[\alpha_1,\alpha_2,-\alpha_3]$ Fréchet derivative
where the $\alpha_j$ are test parameters related to a test set $B$, as in Theorem \ref{thm_Lin_inclusionDetection}.

\begin{algorithm} 
\caption{Linearized reconstruction of  $\osupp(D) \subset \Omega$.}\label{alg_shapeInclusion}
\begin{algorithmic}[1]

\STATE  Choose a collection of sets $\mathcal{B} = \{ B \subset \Omega\}$.

\STATE  Set the approximating collection $\mathcal{A} = \{\}$.

% \STATE  Compute $M_0$ using Theorem  \ref{thm_Lin_inclusionDetection}

\FOR{  $B \in \mathcal{B}$}

\FOR{  $\Lambda^\flat$ with parameters varied as suggested by Theorem \ref{thm_Lin_inclusionDetection}}

\STATE Compute $N_B := \sum_{\sigma_k < 0} 1$, where $\sigma_k$ are the eigenvalues of 
$\Lambda_0 + \Lambda'_0[\alpha_1,\alpha_2,-\alpha_3] - \Lambda$ 

\IF {$ N_B < \infty$}

\STATE 
Add $B$ to the approximating collection $\mathcal{A}$, since 
Theorem \ref{thm_Lin_inclusionDetection} suggests 
\STATE 
that $B \subset \osupp(D)$. 

\ELSE

\STATE Discard $B$, since by Theorem \ref{thm_Lin_inclusionDetection} $B \not \subset D_j$, $j=1,2,3$.

\ENDIF
\ENDFOR
\ENDFOR

\STATE  Compute the union of all elements in $\mathcal{A}$ and all components of 
$\Omega \setminus \cup \mathcal{A}$ not connected 
\STATE to $\p\Omega$. The resulting set is an approximation of $\osupp(D)$.

\end{algorithmic}
\end{algorithm}

\begin{rem}
Some further remarks concerning algorithm \ref{alg_shapeInclusion} are:
\begin{itemize}

\item The linearized shape reconstruction method of  algorithm \ref{alg_shapeInclusion} 
is significantly faster to compute than the procedure in \cite{EP24}.
The speedup is mainly due to that $\Lambda_0 + \Lambda'_0[\alpha_1,\alpha_2,-\alpha_3]- \Lambda$ 
is, for different $\alpha_j$:s and $B$, much faster to compute than 
$\Lambda_{\alpha_1,\alpha_2,\alpha_3}- \Lambda$, that was needed in \cite{EP24}, 
because of formula \eqref{bilinear_Frechet}.\\

\item It should be possible to replace the condition $N_B < \infty$, as in \cite{EP24},
with a finite bound $N_B < M_0$, for some  $M_0 \in \R_+$.  We have however not investigated
this further here.\\ %See also remark \ref{rem_M0}. \\

\item  
Note that one counts the eigenvalues of $\Lambda_0 + \Lambda'_0[\alpha_1,\alpha_2,-\alpha_3]- \Lambda$ 
in algorithm \ref{alg_shapeInclusion} as many times as its multiplicity indicates. \\

\end{itemize}

\end{rem}

\subsection{Reconstructing $\osupp(\mu - \mu_0)$ from $\Lambda$.} \label{sec_mu_recons} 
We can use part (3) of Theorem \ref{thm_Lin_inclusionDetection} to further
improve on the reconstruction we obtain by the method described in the previous section.
Here we will describe how to reconstruct $\osupp(\mu - \mu_0)$.
It will be convenient to set 
$$
\bar \mu := \mu - \mu_0, 	\qquad  \bar \lambda := \lambda - \lambda_0,
$$
where $\lambda$ and  $\mu$ are of the form given by equation \eqref{eq_lambdaMuRho}.
We further assume in this section that
$$
\rho = \rho_0, \qquad  \Gamma_N = \p \Omega.
$$
%
% One could  also consider the case $\osupp(\rho) \cap (\osupp(\mu) \cup \osupp(\lambda)) = \emptyset$,
% but we refrain from doing so.
The first condition says that only the Lamé parameters are perturbed and that 
the background density is constant. The second
condition says that the Dirichlet portion of the boundary is empty. 
It follows that
$$
\osupp(D) = \osupp \big( \supp(\bar \mu) \cup  \supp(\bar\lambda) \big).
$$
We can again reconstruct $\osupp(D)$ by the procedure in subsection \ref{sec_reconstruct_D}.
We will now see how one can reconstruct
the set $\osupp(\bar\mu)$ utilizing part $(3)$ of Theorem \ref{thm_Lin_inclusionDetection}.
Note however that we cannot identify the set $\osupp(\bar \lambda)$, 
unless we know apriori that $\supp(\bar\mu) \cap \supp(\bar \lambda ) \neq \emptyset$. 
% More precisely we cannot tell if 
% $\osupp(\lambda) \subset \osupp(\mu) \cup \osupp(\lambda)$ is a proper subset or not.
Instead we recover only  $\osupp(D) \setminus \osupp(\bar\mu)$.

\medskip
\noindent
To reconstruct $\osupp( \bar \mu)$ we can use essentially the same procedure as for reconstructing $\osupp(D)$
that is outlined in algorithm \ref{alg_shapeInclusion}, 
in section \ref{sec_reconstruct_D}.
One replaces $\osupp(D)$ with $\osupp(\bar \mu)$, and 
performs a similar testing procedure which proceeds as follows.

Let $B \subset \Omega$ be ball. 
To test if $B \subset \osupp(\bar \mu)$ or not, one chooses $\alpha_1=\alpha_3 = 0$,
and $0<\alpha_2 < \gamma_0$, so that  $\alpha = (0,\alpha_2,0)$ in Theorem \ref{thm_Lin_inclusionDetection}.
We can then compute the $\mathcal N$ appearing in Theorem \ref{thm_Lin_inclusionDetection}.
The corresponding $\mathcal{N}$ is either such that $\mathcal{N} = \infty$ or $\mathcal{N} < \infty$.
In the case $\mathcal{N} < \infty$, we know  by part (3) of Theorem \ref{thm_Lin_inclusionDetection},
that it is not the case that $B \not \subset \osupp(\bar \mu)$,
and we classify $B$ as $B \subset \osupp(\bar \mu)$, which is a reasonable approximation for small $B$. 
If  $\mathcal{N} = \infty$  we know by part (1)
of Theorem \ref{thm_Lin_inclusionDetection} that $B \not \subset D_2 = \supp(\bar \mu)$, 
so that  $B$ intersects the complement of $\osupp(\bar \mu)$ or an internal 
cavity\footnote{An internal cavity is here understood as a component 
of $\Omega \setminus \supp(\bar \mu)$ that is not connected to boundary $\p \Omega$.}
of $\supp(\bar \mu)$.
In both cases we  classify $B$ as lying outside $\osupp(\bar \mu)$. 
Notice that if $B$ intersects an internal cavity, then it can be miss classified as lying outside
$\osupp(\bar \mu)$. This 
could leave  cavities in the reconstruction of $\osupp(\bar \mu)$. This can be easily 
corrected, since the set $\osupp(D)$ does not allow for any internal cavities and these can be removed
at the end.  

\section{Numerical simulations} \label{sec_numerics}

\noindent
In this section we deal with different settings and configurations to numerically test the
developed results in the case $\rho=\rho_0$. In more detail, we introduce two different
examples: in the first example (see Subsection \ref{sep_inclu}) we take a look
at an elastic cube with two separated inclusions (cubes), where only one
parameter $\lambda$ or $\mu$ differs from the background and in the second
example we consider two intersecting inclusions (see Subsection
\ref{int_inclu}) so that we have an inclusion only in $\lambda$, an inclusion
only in $\mu$ and in the intersection of both an inclusion in both $\lambda$
and $\mu$.
The material parameters of the elastic body are given in Table \ref{Table_parameter}.
\renewcommand{\arraystretch}{1.4}
\begin{table}[h!]
\begin{center}
\begin{tabular}{ |c|c| c | c |}  
\hline
material & $\lambda_i$ & $\mu_i$  & $\rho_i$ \\
\hline
$i=0$: background &  $6\cdot 10^5$   &  $6\cdot 10^{3}$  &  $3\cdot 10^3$\\
 \hline
$i=1$: inclusion &  $2\cdot 10^{6}$ &  $2\cdot 10^{4}$  &   $3\cdot 10^3$\\
\hline
\end{tabular}
\end{center}
\caption{Lam\'e parameter $\lambda$ and $\mu$  in [$Pa$] and density $\rho$ in [$kg/m^3$].}
\label{Table_parameter}
\end{table}
\noindent
In the following, we examine the frequency
$\omega=50\frac{rad}{s}$. The corresponding p-wave and s-wave wavelengths\footnote{The
s-wavelength and p-wavelength are defined via $l_p=2\pi\dfrac{v_p}{k}$ and
$v_s=2\pi\dfrac{v_s}{k}$ with the velocities
$v_p=\sqrt{\dfrac{\lambda_0+2\mu_0}{\rho_0}}$ and
$v_s=\sqrt{\dfrac{\mu_0}{\rho_0}}$} are $l_p=1.70$ m and $l_s=0.18$ m. If not
otherwise mentioned, we  divide each face of the cube into $10\times 10$
patches resulting in 600 Neumann patches (see Figure \ref{patches_10x10}). On
each, we apply boundary loads in all three directions (normal and tangential
directions) resulting in 1800 different boundary loads. Further, the Dirichlet
boundary $\Gamma_D=\emptyset$.

\begin{figure}[h]
\centering
\includegraphics[scale=0.4]{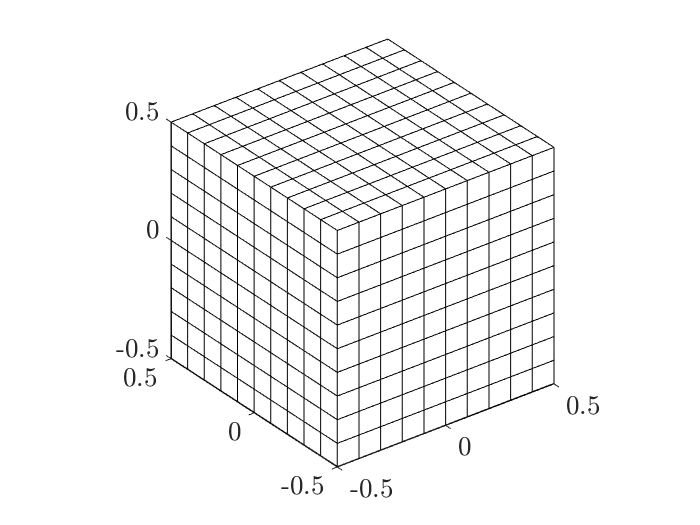}
\caption{Illustration of the Neumann boundary consisting of $600$ patches ($10\times 10$ on each surface).}\label{patches_10x10}
\end{figure}
\noindent
\\
We follow the algorithms form Subsection \ref{sec_reconstruct_D} and
\ref{sec_mu_recons}. Recall from Subsection \ref{sec_mono_shape} that
$$
D_1 = \supp(\lambda - \lambda_0), \qquad 
D_2 = \supp(\mu - \mu_0).
$$
Furthermore, we use the following color-coding for the reconstructions:
\begin{itemize}
\item red: test with $\alpha_1= C(\lambda-\lambda_0)=18760$ and $\alpha_2=\alpha_3=0$ resulting in the reconstructed set $D=D_1\cup D_2$
\item blue: test with $\alpha_2=2C(\mu-\mu_0)= 375.2$ and $\alpha_1=\alpha_3=0$ resulting in the reconstructed set $ D_2$
\item green: set difference $D\setminus D_2$
\end{itemize}
\noindent
Here, the constant $C=0.0134$ was determined experimentally and remains
unchanged over all tests in this chapter regardless of the test model, the
number of boundary loads or the number of test inclusions. 

\smallskip
\noindent
Algorithm 1 from Subsection \ref{sec_reconstruct_D} determines, if a test block
lies inside $D=D_1\cup D_2$, if the number of negative eigenvalues $N_B$ of
$\Lambda_0+\Lambda_0’ [\alpha_1,\alpha_2,-\alpha_3]-\Lambda$ is finite. However,
since the numerical setup can only handle finite data, this is true for all
test inclusions by default. Hence, we choose a threshold $\tilde{M}_1$ and
accept a test inclusion $B$ to be inside the inclusion, if $N_B\leq
\tilde{M}_1$ similarly as in \cite{EP24}. We do the same for the test
regarding $D_2$ by choosing $N_B\leq \tilde{M}_2$.

\subsection{Separated inclusions}\label{sep_inclu}
The model we consider is given on the left in Figure \ref{test_model}, 
where the magenta inclusion stands for $\mu$ and the yellow
one for $\lambda$. All in all, the model consists of two blocks with block
centers aligned on the diagonal from $\frac{1}{2}(-1,-1,-1)^T$ to
$\frac{1}{2}(1,1,1)^T$.

\begin{figure}[h]
\centering
\includegraphics[scale=0.2]{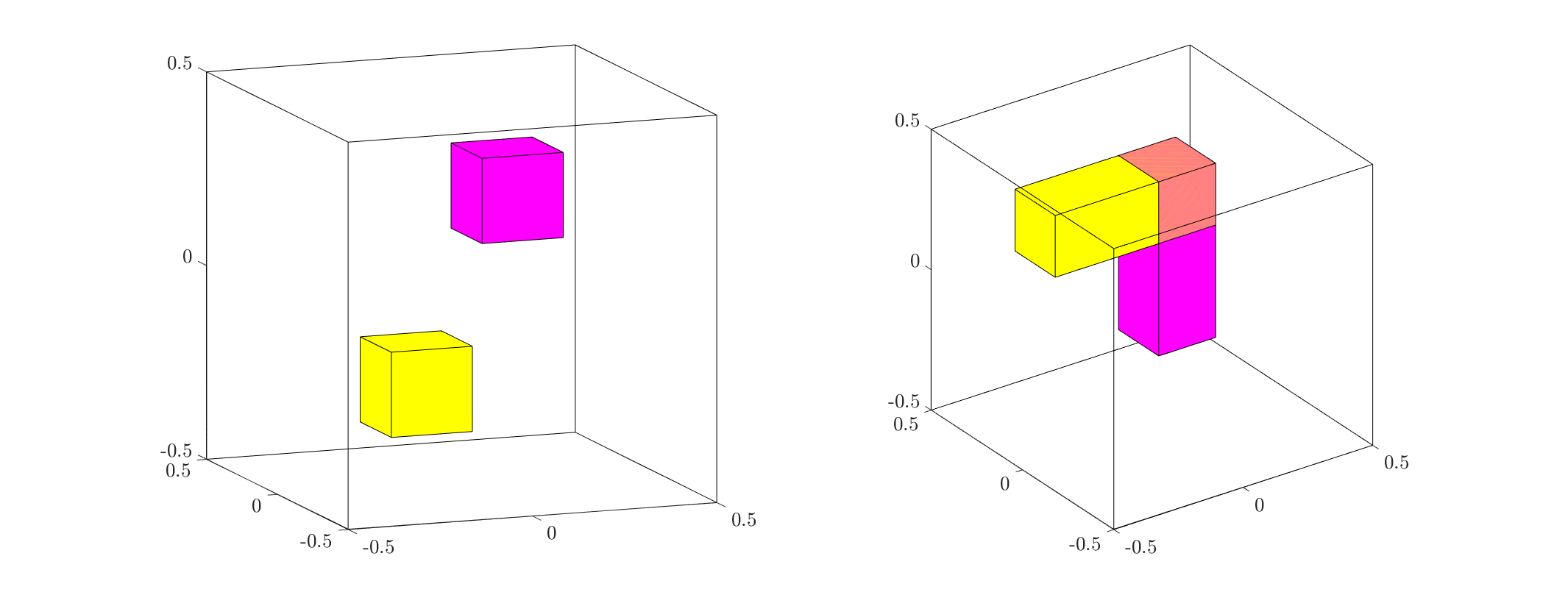}
\caption{Test model with two separated inclusions (left) and intersecting rods (right). Magenta: $D_2\setminus D_1$, yellow $D_1\setminus D_2$ and orange the inclusion $D_1\cap D_2$.}\label{test_model}
\end{figure}
\noindent
We perform the testing with $1000$ ($10\times 10\times 10$) test cubes.
As can be seen, we can reconstruct both inclusions via our choice of
$\tilde{M}_1=408$ as shown in Figure \ref{cubes_50_Hz_1000_blocks} (first row:
left). Since the choice on first glance seems to be arbitrary, we motivate it
by taking a look at $N_B$ for each of the $1000$ test blocks depicted in Figure
\ref{cubes_50_Hz_1000_blocks} (second row: left). It can clearly be seen, that
the number of negative eigenvalues $N_B$ of the marked test inclusions differs
significantly from those test blocks outside of the inclusions, which seem all
to be contained inside a certain range regardless of location. The same can be
said for the reconstruction of $D_2$ depicted in Figure
\ref{cubes_50_Hz_1000_blocks} (first row: middle). The choice of $\tilde{M}_2=410$ is again motivated by a look at the
eigenvalue plot in Figure \ref{cubes_50_Hz_1000_blocks} (second row, right). All
in all, the set difference $D\setminus D_2$ results in the
reconstruction of the set $D_1$ as depicted in Figure
\ref{cubes_50_Hz_1000_blocks} (first row, right).

\begin{figure}[h]
\centering
\includegraphics[scale=0.25]{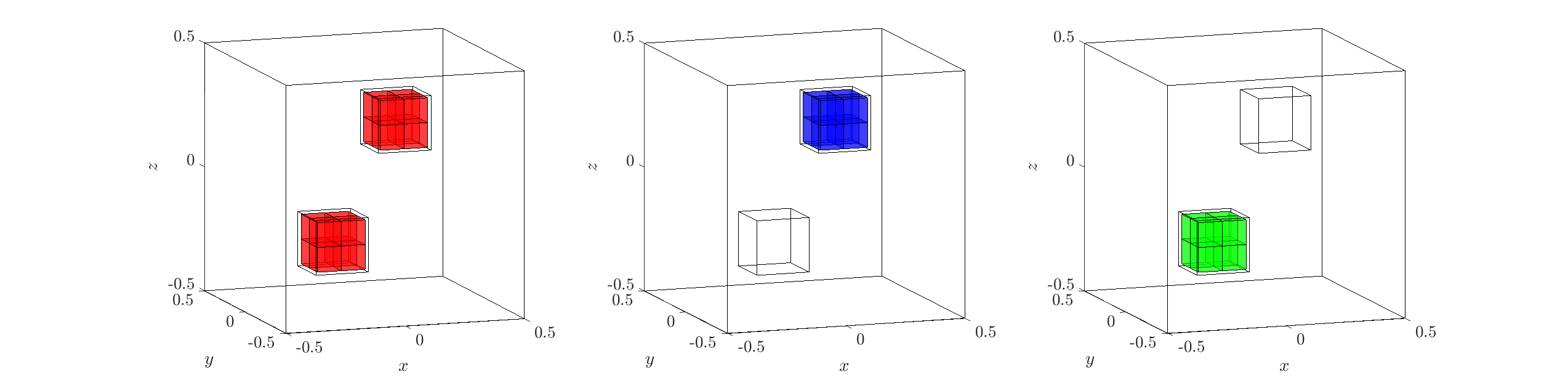}\\
\includegraphics[scale=0.3]{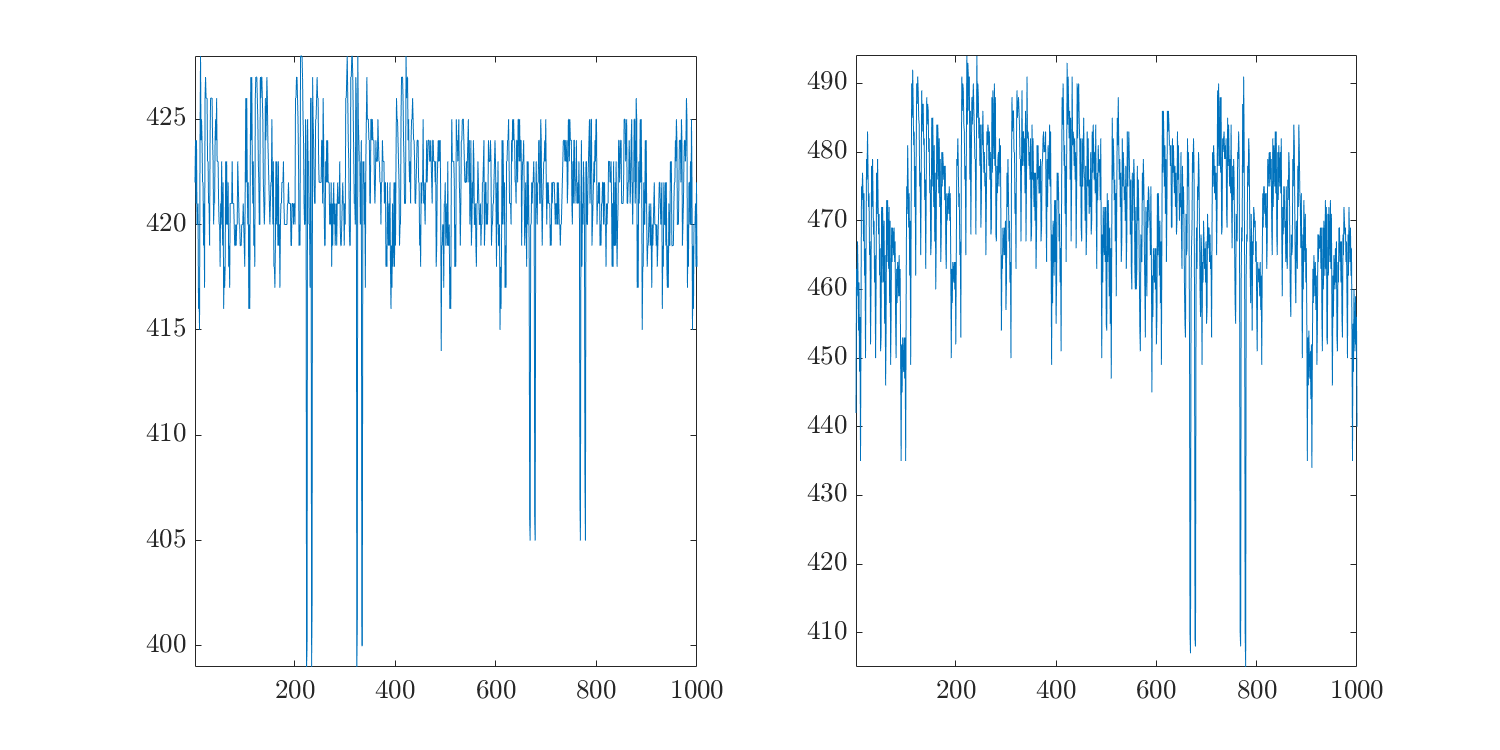}
\caption{We consider the case $\omega = 50 \frac{rad}{s}$ and $1000$ test
blocks. First row: Reconstruction of the inclusions: left: $D=D_1\cup D_2$ for
$\tilde{M}_1 = 408$ (red), middle: $D_2$ for $\tilde{M}_2=410$, right:
$D\setminus D_2$. Second row: Plot of the number of negative eigenvalues for
each test block: left: $D=D_1\cup D_2$, right:
$D_2$.}\label{cubes_50_Hz_1000_blocks}
\end{figure}

\begin{rem}
It should be noted, that choosing a singular $M=410$ instead of separate
$\tilde{M}_1$, $\tilde{M}_2$ would also result in the same reconstruction.
\end{rem}

\subsection{Intersecting inclusions}\label{int_inclu}
Next, we deal with two intersecting inclusions (two rods) as depicted in Figure \ref{test_model}, right.
\\
\\
Figure \ref{overview} shows the reconstruction of $D_1\cup D_2$ for $1000$ test
blocks and $\omega= 50\ \frac{rad}{s}$. It can be seen that all test blocks in
the inclusion are marked correctly, however several test blocks were marked as
inside the inclusion not contained in $D$. Those blocks lie either
in the nook between the two rods, or between boundary and inclusion. This seems
to be a common trait for the elastic monotonicity methods and was already
observed in \cite{EH22} as well as \cite{EH23}. This can be handled by taking
more boundary loads into account in order to better discretize the operators
$\Lambda$ and $\Lambda_0$.

\begin{figure}[h]
\centering
\includegraphics[scale=0.25]{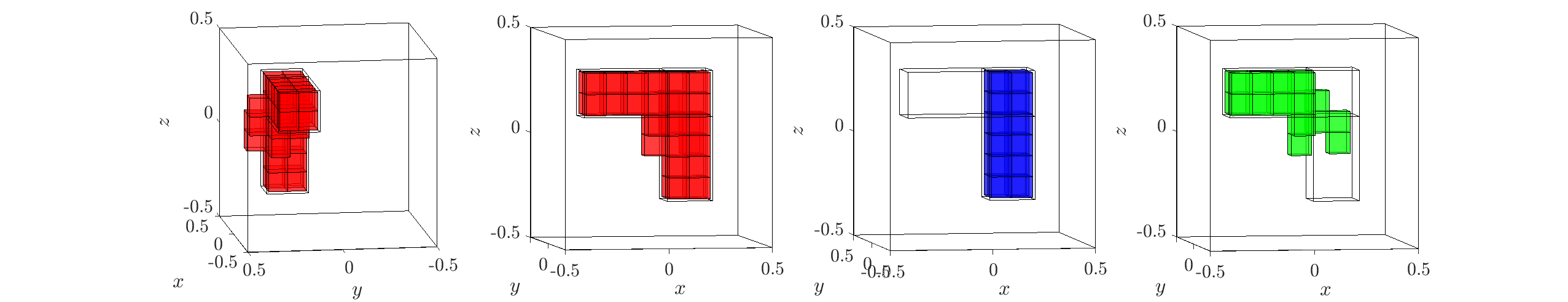}
\caption{Reconstruction for $\omega = 50 \frac{rad}{s}$ with $1000$ test
blocks. $D=D_1\cup D_2$ and $\tilde{M}_1 = 255$ (red), $D_2 $ and $\tilde{M}_2
= 267$ (blue) and $D_1\setminus D_2$ (green).}\label{overview}
\end{figure}
\noindent
The reconstruction of $D_2$ in Figure \ref{overview} is based on
part (3) of Theorem \ref{thm_Lin_inclusionDetection}. This results in a correct reconstruction of $D_2$. All in
all, taking the set difference results in the set $D\setminus D_2$ plus the
blocks assigned wrongly in Figure \ref{overview}.

\subsection{Normal and tangential boundary loads}\label{diff_boundary}
As a last test, we take a look at only taking boundary loads in normal or
tangential direction into account. The reason for this is simple. By reducing
the amount of boundary loads, a lab experiment can be set up easier, we reduce
the computation time significantly and applying only normal loads is physically
easier than applying tangential loads. As a test model, we take the two
separate inclusions (Figure \ref{test_model}, left) and $1000$ test inclusions,
as well as $\omega=50 \frac{rad}{s}$.

\begin{figure}[h]
\centering
\includegraphics[scale=0.25]{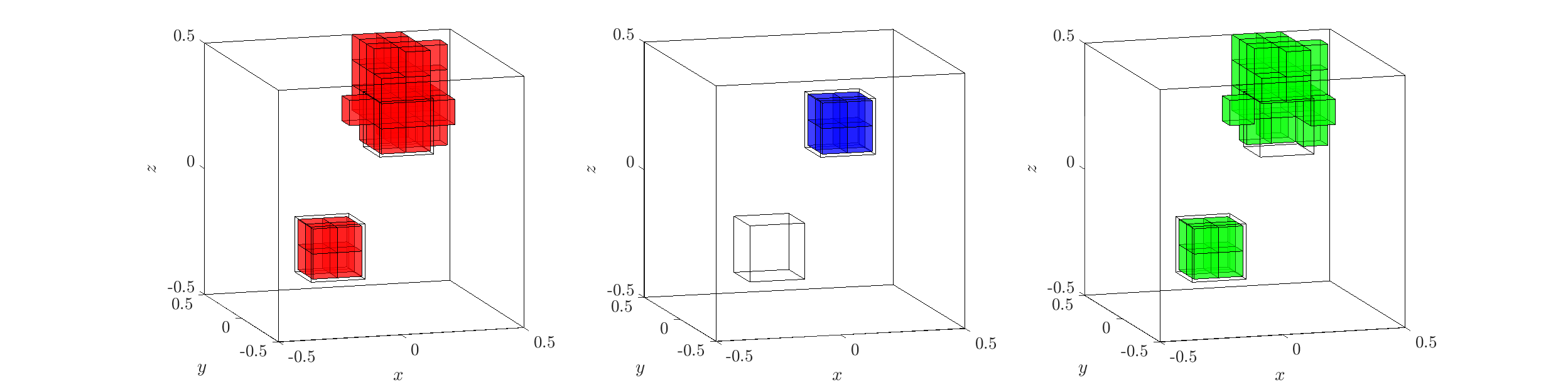}
\caption{Reconstruction for $\omega = 50 \frac{rad}{s}$ with $1000$ test blocks
and $600$ normal boundary loads. $D=D_1\cup D_2$ and $\tilde{M}_1 = 28$ (red),
$D_2 $ and $\tilde{M}_2 = 28$ (blue) and $D_1\setminus D_2$
(green).}\label{overview_normal}
\end{figure}
\noindent
As depicted in Figure \ref{overview_normal}, normal loads are enough to
reconstruct $D_2$ (blue), but the reconstruction of $D=D_1\cup D_2$ (red) is
significantly worse than in Figure \ref{cubes_50_Hz_1000_blocks} for $600$
normal loads and $1200$ tangential loads used together resulting in a bad
reconstruction of $D_1$ via the set difference (green). 
\\
\\
On the contrary, using only tangential loads on the same surface patches, we obtain
the same reconstruction as in Figures \ref{cubes_50_Hz_1000_blocks}. This may
have two reasons. First, using more boundary loads in general leads to a better
reconstruction and second, using tangential loads introduce a higher rotational
part in $u$ which seems to favor the reconstruction.
\\
\\
These initial tests of the monotonicity method result, all in all,
% All in all, the monotonicity methods result 
in a good reconstruction of
inclusions and are to a degree able to separate between inclusions in
$\lambda$ and $\mu$, when $\rho=\rho_0$. % up to intersections, which will be included in $D_2$ only.
Further it is possible to reduce computation time by only taking normal
or tangential boundary loads into account but it should be noted that using tangential parts
is preferred to normal loads for a better reconstruction. If computation time
is not an issue, the use of tangential and normal loads is preferred for the
best reconstruction ability. 
\\
\\

\section{Appendix}

\subsection{Estimates for a source problem.} \label{sec_ellipticEst}
Here we investigate the well-posedness  and apriori estimates, for the weak solutions 
to the source problem
$$
u \in \mathcal{V} :=  \{ u \in H^1(\Omega)^3 \,:\, u|_{\Gamma_D} = 0 \}
$$
to the formal boundary value problem
\begin{align}  \label{eq_bvp2}
\begin{cases}
\nabla \cdot (\C\,  \hat \nabla u )  + \omega^2\rho u + \tau u &=  F + \nabla \cdot A \\
% \sum_{k=1}^3 \p_k F^k \\ %\phi_1F^1 + \nabla \cdot (\phi_2 (\nabla \cdot F^2)I)+\nabla\cdot (\phi_3 \hat \nabla F^3), \\
\;\quad\quad\quad\quad(\gamma_\C u ) |_{\Gamma_N} &=  A \nu |_{\Gamma_N} , \\	
 \quad\quad\quad\quad\quad \quad u |_{\Gamma_D} &= 0.	
\end{cases}
\end{align}
where $F \in L^2(\Omega)^3$ and $A \in L^2( \Omega)^{3 \times 3}$. %, with $\supp(A) \Subset \Omega$. 
In general  $ \nabla \cdot A$ will lie in the dual space  $\mathcal{V}^*$, and $\tau \in \R$.
The norm in the dual space is given by
\begin{equation} \label{eq_Vdual_norm}
\begin{aligned}
\| \ell \|_{\mathcal{V}^*} := \sup \big\{ |l(v)| \,:\, v \in \mathcal{V}, \,\| v \|_{H^1(\Omega)^3} =1 \big\}.
\end{aligned}
\end{equation}
A weak solution to \eqref{eq_bvp2} is an element $u \in \mathcal V$ such that
\begin{align} \label{eq_weak3}
B_\tau(u,v) := B(u,v) + \tau (u,v)_{L^2(\Omega)^3}  = \int_\Omega F \cdot v \,dx - \int_{\Omega}  A : \hat \nabla v \,dx,
% + \int_{\Gamma_N} A \nu \cdot v\, dS,
\qquad \forall v \in \mathcal V,
\end{align}
where $B$ is as in \eqref{eq_weak2}.
To show that \eqref{eq_bvp2} is well-posed, outside a discrete set of $\tau \in \R$, 
we will need the following Lemma.

\begin{prop} \label{prop_elipEst_1}
% There exists a numerable set $\{\tau_k \in \R \;:\; k \in \N\}$, s.t.
There exists a $\tau_0 \leq 0$, for which
the boundary value problem in \eqref{eq_bvp2} admits a unique weak solution
$u \in \mathcal{V}$, which satisfies
\begin{align}  \label{eq_aprioriEst}
\| u \|_{ H^1(\Omega)^3} 
\leq C 
(\| F \|_{L^2( \Omega)^{3}}+\| A \|_{L^2( \Omega)^{3 \times 3}}).
\end{align}
\end{prop}

\begin{proof} 
We use the Lax-Milgram Lemma to prove the uniqueness and existence of a weak solution
(see e.g. Theorem 1.3 in \cite{OSY92}).
We need to show that the bilinear form $B_\tau: \mathcal{V}\times \mathcal{V}\to \R$ 
is coercive and continuous for some $\tau \leq 0$. For coercivity we need 
to show that 
\begin{align} \label{eq_Btau_coer}
|B_{\tau_0}(u,u)| \geq c \| u \|^2_{ H^1(\Omega)^3}, \quad c>0,
\end{align}
for some $\tau_0 \leq 0$. Let $\tau \leq 0$.
The Korn inequality (see Theorem 2.4 p 17 in \cite{OSY92}) gives that
\begin{align*} 
|B_\tau(u,u)| &\geq 
\int_\Omega 2 \mu \hat \nabla u :\hat \nabla u + \lambda \nabla \cdot u \nabla \cdot u 
-\tau \int_\Omega u^2\,dx
- \omega^2 \| \rho \|_{L^\infty}\int_\Omega u^2\,dx \\ %\| u \|^2_{L^2(\Omega)^3} \\
&\geq
C \big(\| \hat \nabla u \|^2_{L^2(\Omega)^{3\times 3}} 
+\| \nabla \cdot u \|^2_{L^2(\Omega)} \big)
-(\tau + \omega^2\| \rho \|_{L^\infty})\| u \|^2_{L^2(\Omega)^3} \\
&\geq
C \| u \|^2_{H^1(\Omega)^3} 
-(C+\tau + \omega^2\| \rho \|_{L^\infty})\| u \|^2_{L^2(\Omega)^3}. 
\end{align*}
Thus if we choose a $|\tau|$ large enough with $\tau\leq 0$, the 2nd term is positive and can be dropped.
We thus see that we can choose a $\tau_0 \leq 0$  so that \eqref{eq_Btau_coer} holds.

The continuity of $B_\tau$ follows from the estimate
\begin{align*} 
|B_\tau(u,v)| &=  
\Bigg|\int_\Omega
2 \mu \hat \nabla u :\hat \nabla v + \lambda \nabla \cdot u \nabla \cdot v 
- (\tau + \omega^2\rho) u\cdot v\,dx
\Bigg| \\
&\leq
C (\| \hat \nabla u \|_{L^2(\Omega)^{3\times 3}} \| \hat \nabla v \|_{L^2(\Omega)^{3\times 3}}
+\| \nabla \cdot u \|_{L^2(\Omega)} \| \nabla \cdot v \|_{L^2(\Omega)}
+\| u \|_{L^2(\Omega)^3} \| v \|_{L^2(\Omega)^3}) \\
&\leq
C \| u \|_{H^1(\Omega)^3} \| v \|_{ H^1(\Omega)^3}.
\end{align*}
The Lax-Milgram Lemma gives us thus the existence of a unique $u \in \mathcal{V} \subset H^1(\Omega)^3$
for which \eqref{eq_weak2} holds for all $v \in \mathcal{V}$, and when $\tau \leq 0$, $|\tau|$ is large enough.
We have thus found a weak solution in accordance with \eqref{eq_weak2}.

The final step is to verify the estimate of the claim.
By the coercivity of the form $B$, we have that
\begin{small}
\begin{align*}
\| u \|^2_{H^1(\Omega)^3} \leq  C |B(u,u)| &= \Big | \int_\Omega F \cdot u \,dx - \int_{\Omega}  A : \hat \nabla u \,dx \Big|
&\leq C ( \| F \|_{L^2( \Omega)^{3}}+\| A \|_{L^2( \Omega)^{3 \times 3}}) \| u \|_{H^1(\Omega)^3}.
% \| u \|^2_{H^1(\Omega)^3} \leq  C |B(u,u)| = C |\ell(u)| \leq C \| \ell \|_{\mathcal{V}^*} \| u \|_{H^1(\Omega)^3}.
\end{align*}
\end{small}
\end{proof}

\begin{prop} \label{prop_elipEst_1}
Assume that zero is not an eigenvalue for the mixed eigenvalue problem related to \eqref{eq_bvp2}.
Then \eqref{eq_bvp2} admits a unique weak solution $u \in \mathcal{V}$, which satisfies
\begin{align}  \label{eq_aprioriEst2}
\| u \|_{ H^1(\Omega)^3} 
\leq C 
% \| \ell \|_{\mathcal{V}^*}.
(\| F \|_{L^2( \Omega)^{3}}+\| A \|_{L^2( \Omega)^{3 \times 3}}).
\end{align}
\end{prop}

\begin{proof}
According to Proposition 3.3 in \cite{EP24} the spectrum of the mixed eigenvalue problem related to \eqref{eq_bvp2}
consists of eigenvalues
$$
\sigma_1 \geq \sigma_2 \geq \sigma_3 \geq \dots \to -\infty.
$$
Assume that  $\sigma_k \neq 0$ for all $k$, so that zero is not an eigenvalue and we can 
choose $\tau =0$.  And in particular there is no $\tilde u \not \equiv 0$, such that
\begin{equation} \label{eq_no_ev}
\begin{aligned}
B(\tilde u, v) = 0, \qquad \forall v \in \mathcal{V}.
\end{aligned}
\end{equation}
Let $\tau_0$ be as in Proposition \ref{prop_elipEst_1}, and define the solution map
$$
K_{\tau_0}: (F,A) \mapsto u,\qquad K_{\tau_0}: L^2(\Omega)^3 \times L^2(\Omega)^{3 \times 3} \to H^1(\Omega)^3,
$$
where $u,F$ and $A$ are as in Proposition \ref{prop_elipEst_1}.
A vector field $u$ is a weak solution to problem \eqref{eq_bvp2}, if 
\begin{align*}
& B(u,v) = (F,v)_{L^2(\Omega)^3} - (A,\hat \nabla v)_{L^2(\Omega)^{3 \times 3}}, \quad \forall v \in \mathcal{V} \\ 
&\;\Leftrightarrow\;
B_{\tau_0}(u,v) = (F,v)_{L^2(\Omega)^3} - (A,\hat \nabla v)_{L^2(\Omega)^{3 \times 3}}+ \tau_0(u,v)_{L^2}, \quad \forall v \in \mathcal{V} \\
&\;\Leftrightarrow\; 
u = K_{\tau_0}(F + \tau_0 u, A) \\
&\;\Leftrightarrow\;
u - \tau_0 K_{\tau_0}(u, 0) = K_{\tau_0}(F,A) \\
&\;\Leftrightarrow\;
(I - \tau_0 \tilde K_{\tau_0})u = \tilde F,
\end{align*}
where $\tilde F :=  K_{\tau_0}(F,A)$ and $\tilde K_{\tau_0} := K_\tau( \cdot , 0)$.
The mapping $\tilde K_{\tau_0}:  L^2(\Omega)^3 \times L^2(\Omega)^{3 \times 3}  \to L^2(\Omega)^3$ is compact, 
which can be seen using the Sobolev embedding. 
We can thus interpret the last line as  an integral equation over $L^2$.
By the Fredholm alternative $(I - \tau_0 \tilde K_{\tau_0})u = F$ has a solution if and only if 
the corresponding homogeneous equation has no non  trivial solutions. 
There is no non trivial solutions to the homogeneous equation, because of \eqref{eq_no_ev}.
And hence we have a unique weak solution $u$ by the above equivalences.
% \medskip
% NEW END 
% \medskip

\end{proof}

\subsection{Existence and uniqueness for a second order Maxwell type equation.} \label{sec_maxwell}
In this subsection we will derive some results on the existence and uniqueness of solutions
to equation \eqref{eq_maxwell} using Fredholm theory. These are well known and appear in various forms elsewhere. 
We however present them in this part of the appendix in the specific forms needed for this paper. 

\begin{lem} \label{lem_maxwell_source1}
There exists a discrete set $\mathcal Z  \subset \R$, such that the source problem
\begin{equation} \label{eq_maxwell_source}
\begin{aligned}
-\nabla \times \nabla \times u+ \omega^2 \frac{\rho_0}{\mu_0} u &= F,  \quad \text{ in } \Omega, \\
(\nu \times  u)|_{\p\Omega} &= 0,  \quad \text{on } \p \Omega, 
\end{aligned}
\end{equation}
has a unique solution $u \in H^1(\Omega)^3$, when $\omega \notin \mathcal Z$, and
where $F \in L^2(\Omega)^3$, with $\nabla \cdot F = 0$. Moreover $ \nabla \cdot u= 0$.
\end{lem}

\begin{proof}
The claim follows from  Theorem A.1 \cite{Zh10}.
\end{proof}

\begin{lem} \label{lem_maxwell_bc1}
There exists a discrete set $\mathcal Z  \subset \R$, such that the boundary value problem
\begin{equation} \label{eq_maxwell_bc1}
\begin{aligned}
-\nabla \times \nabla \times u+ \omega^2 \frac{\rho_0}{\mu_0} u &= 0,  \quad \text{ in } \Omega, \\
(\nu \times  u)|_{\p\Omega} &= f,  \quad \text{on } \p \Omega, 
\end{aligned}
\end{equation}
has a unique solution $u \in H^1(\Omega)^3$, when $\omega \notin \mathcal Z$, and
where $f \in H^{1/2}_t(\p \Omega)$. Furthermore $\nabla \cdot u= 0$.
\end{lem}

\begin{proof}
The claim follows from  Theorem A.1 \cite{Zh10}.
\end{proof}

\begin{lem}  \label{lem_maxwell_source3}
Assume that $ \omega \neq 0$.
There exists a subspace $\mathcal E \subset H^1(\Omega)^3 $  of finite  dimension, such that
the source problem
\begin{equation} \label{eq_maxwell_source_3}
\begin{aligned}
-\nabla \times \nabla \times u+ \omega^2 \frac{\rho_0}{\mu_0} u &= F,  \quad \text{ in } \Omega, \\
(\nu \times  u)|_{\p\Omega} &= 0,  \quad \text{on } \p \Omega, 
\end{aligned}
\end{equation}
has a solution $u \in H^1(\Omega)^3$, when $F \in L^2(\Omega)^3$, with $ \nabla \cdot F = 0$, provided %and only if 
that
$$
(F,\phi)_{L^2(\Omega)^3} = 0, \qquad \forall \phi \in \mathcal E.
$$
The solution $u$ is unique as an element of $H^{1}(\Omega)^3 / \mathcal{E}$.
Moreover $ \nabla \cdot u= 0$.  
\end{lem}

\begin{proof}
We first note that the assumption $\omega \neq 0$, implies that a solution $u$ to \eqref{eq_maxwell_source_3}
is such that $ \nabla \cdot u = 0$, and hence in $u \in L^2(\operatorname{div}\,0; \Omega)$, when $\omega \neq 0$.
We will use Fredholm theory. To this end define 
$$
K_0 : F \mapsto u, \qquad K_0 : L^2(\operatorname{div}\,0; \Omega) \to L^2(\operatorname{div}\,0; \Omega),
$$
where $u$ solves \eqref{eq_maxwell_source}, and where  in Lemma \ref{lem_maxwell_source1} we choose the $\omega$
as $\omega_0$ and where $\omega_0$ is such that a solution exists. 
The map $K_0$ is compact. To see this note first that $L^2(\operatorname{div}\,0; \Omega) \subset L^2(\Omega)^3$ is a closed
subspace. 
Note also that $u = K_0F \in H^1(\Omega)^3$, since  the equation \eqref{eq_maxwell_source} 
implies that $u \in \Hcurl{\Omega}$. And moreover since we have that
$$
\| u \|_{H^1(\Omega)^3} \sim
\| u \|_{L^2(\Omega)^3}
+\| \nabla \cdot	u \|_{L^2(\Omega)}
+\| \nabla \times u \|_{L^2(\Omega)^3}
+ \| \nu \times u \|_{H^{1/2}(\p \Omega)}	 
$$
according to corollary 5 on p. 51 in \cite{Ce96}.
Now $K_0 F \in H^1(\Omega)^3$, so that the compactness of the inclusion $H^1(\Omega)^3\subset L^2( \Omega)^3$
and continuity of the projection of  $L^2(\Omega)^3$ to $L^2(\operatorname{div}\,0; \Omega)$,
implies that $K_0$ is compact.
We claim that the operator $K_0$ is moreover self-adjoint.
Let us first check that $K_0$ is symmetric.
It is enough to check symmetry with respect to the $L^2$-inner product in a dense set.
Assume hence that $F,G \in \C^\infty_0(\Omega)^3 \cap L^2(\operatorname{div}\,0; \Omega)$.
By the weak formulation of \eqref{eq_maxwell_source_3}, we see that
$$
(K_0F, G)_{L^2(\Omega)^3} 
=-( \nabla \times K_0F, \nabla \times K_0G)_{L^2(\Omega)^3} + \omega^{2}_0 (K_0F,K_0G)_{L^2(\Omega)^3} 
= (F, K_0 G)_{L^2(\Omega)^3}, 
$$
$K_0$ is thus symmetric and bounded and thus self-adjoint.

Next we will apply the Fredholm alternative. We have that 
\begin{align*}
u \text{ solves equation \eqref{eq_maxwell_source_3}}
\quad 
&\Leftrightarrow \quad -\nabla \times \nabla \times u+ \omega_0^2
\tfrac{\rho_0}{\mu_0} u = F + (\omega^2_0 - \omega^2) \tfrac{\rho_0}{\mu_0} u
\\
&\Leftrightarrow \quad (I - (\omega^2_0 - \omega^2)\tfrac{\rho_0}{\mu_0} K_0) u = K_0 F, 
\end{align*}
where $u \in L^2(\operatorname{div}\,0; \Omega)$.
The operator $I - (\omega^2_0 - \omega^2)\tfrac{\rho_0}{\mu_0} K_0$ is a Fredholm operator of index
zero, according to Theorem 2.22 in \cite{Mc00}.
By the Fredholm alternative, see \cite{Mc00} Theorem 2.27, we now know that either  
$$
(I - (\omega^2_0 - \omega^2)\tfrac{\rho_0}{\mu_0} K_0) u = K_0 F, 
$$
has a unique solution $u \in L^2(\operatorname{div}\,0; \Omega)$, for every $K_0 F \in L^2(\operatorname{div}\,0; \Omega)$,
in which case we choose $\mathcal E = \{0\}$, and the claim holds, or the homogeneous equation
\begin{equation} \label{eq_homogeneous}
\begin{aligned}
(I - (\omega^2_0 - \omega^2)\tfrac{\rho_0}{\mu_0} K_0) \phi_k = 0, 
\end{aligned}
\end{equation}
has the non trivial solutions $\phi_1 ,..,\phi_p \in H^1(\Omega)^3$, $p \geq 1$.
In the latter case, we furthermore have that
there is a solution $u$ to 
$$
(I - (\omega^2_0 - \omega^2)\tfrac{\rho_0}{\mu_0} K_0) u = K_0 F, 
$$
provided that 
$$
K_0 F \perp \spn \{ \phi_1 ,..,\phi_p\}
\qquad \Leftrightarrow \qquad
F \perp \spn \{ \phi_1 ,..,\phi_p\},
$$
where we used the symmetry of $K_0$.
In this case we thus set $\mathcal E := \spn \{ \phi_1 ,..,\phi_p\}$, to see that the claim holds.

To prove uniqueness assume that $u_1$ and $u_2$ both solve \eqref{eq_maxwell_source_3} for a given $F$.
Then $u_1 - u_2$ also solves the homogeneous equation \eqref{eq_homogeneous}, so that
$u_1- u_2 = \phi \in \mathcal E$.
\end{proof}	

\noindent
Note that if $\mathcal E \neq \{0\}$ in Lemma \ref{lem_maxwell_source3}, then
$\mathcal{E}$ is the set of eigenfunctions of zero related to the problem 
\eqref{eq_maxwell_source_3}.

\begin{lem}  \label{lem_maxwell_bc3}
Assume that $\omega \neq 0$.
There exists a subspace $\mathcal E \subset H^1(\Omega)^3$  of finite  dimension, such that
the source problem
\begin{equation} \label{eq_maxwell_bc_3}
\begin{aligned}
-\nabla \times \nabla \times u+ \omega^2 \frac{\rho_0}{\mu_0} u &= 0,  \quad \text{ in } \Omega, \\
(\nu \times  u)|_{\p\Omega} &= f,  \quad \text{on } \p \Omega, 
\end{aligned}
\end{equation}
has a solution $u \in H_t^1(\Omega)^3$, when $f \in H_t^{1/2}(\p \Omega)$, provided  that
\begin{equation} \label{eq_f_cond}
\begin{aligned}
( f, \nabla \times \phi)_{L^2( \p \Omega)^3} = 0, \qquad \forall \phi \in \mathcal E.
\end{aligned}
\end{equation}
The solution $u$ is unique as an element of $H^{1}(\Omega)^3 / \mathcal{E}$.
Moreover $ \nabla \cdot u= 0$.  
\end{lem}

\begin{proof}
We will use Lemma \ref{lem_maxwell_source3} to prove the claim.
First we extend $f$ as follows. By Lemma \ref{lem_maxwell_bc1}, there exists
an $\omega_0 \neq 0$ and an $\tilde u \in H^1(\Omega)^3$, that solves the problem
\begin{equation*}% \label{eq_maxwell_bc1}
\begin{aligned}
-\nabla \times \nabla \times \tilde u+ \omega_0^2 \frac{\rho_0}{\mu_0} \tilde u &= 0,  \quad \text{ in } \Omega, \\
(\nu \times  \tilde u)|_{\p\Omega} &= f,  \quad \text{on } \p \Omega. 
\end{aligned}
\end{equation*}
We can assume that $\omega \neq \omega_0$. Now according to Lemma \ref{lem_maxwell_source3}
there exists a subspace $\mathcal E \subset H^1(\Omega)^3$  of finite  dimension, 
and a solution $v \in H^1(\Omega)^3$
such that
\begin{equation*}% \label{eq_maxwell_bc1}
\begin{aligned}
-\nabla \times \nabla \times v + \omega^2 \frac{\rho_0}{\mu_0} v &= 
\nabla \times \nabla \times \tilde u- \omega^2 \frac{\rho_0}{\mu_0} \tilde u =  (\omega_0^2 - \omega^2) \frac{\rho_0}{\mu_0} \tilde u,
\quad \text{ in } \Omega, \\
(\nu \times  v)|_{\p\Omega} &= 0,  \quad \text{on } \p \Omega, 
\end{aligned}
\end{equation*}
provided that 
\begin{equation} \label{eq_F_cond_2}
\begin{aligned}
(\tilde u ,\phi)_{L^2(\Omega)^3} = 0, \qquad \forall \phi \in \mathcal E.
\end{aligned}
\end{equation}
It follows that $ u:= v + \tilde u$ solves \eqref{eq_maxwell_bc_3}. Moreover 
\begin{align*}
\omega_0^2 \tfrac{\rho_0}{\mu_0} ( \tilde u ,\phi)_{L^2(\Omega)^3}
&=
-(\nabla \times \tilde u, \nabla \times \phi )_{L^2(\Omega)^3} + \omega_0^{2}\tfrac{\rho_0}{\mu_0} ( \tilde u,\phi)_{L^2(\Omega)^3}  \\
&=
( \nu \times \tilde u,  \nabla \times \phi )_{L^2( \p \Omega)^3} \\
&=
( f ,  \nabla \times \phi )_{L^2( \p \Omega)^3}, 
\end{align*}
which shows that the condition \eqref{eq_F_cond_2} is implied by \eqref{eq_f_cond}.
\end{proof}

\bigskip
\noindent
{\bf{Acknowledgement}} \\
SEB would like to thank the German Research Foundation (DFG) for funding the
project 'Inclusion Reconstruction with Monotonicity-based Methods for the
Elasto-oscillatory Wave Equation' (Reference Number 499303971). 
VP would like to thank the Research Council
of Finland (Flagship of Advanced Mathematics for Sensing, Imaging and Modelling grant
359186) and the Emil Aaltonen Foundation.

\noindent
\\
\\

\begin{thebibliography}{+}

\bibitem[AG23]{AG23}A. Albicker, R. Griesmaier,
Monotonicity in inverse scattering for Maxwell’s equations, Inverse Probl. and Imaging, 17, 68-1
05, 2023.\\

\bibitem[AK02]{AK02}
C. J. S. Alves, R. Kress, On the far-field operator in elastic obstacle scattering,
IMA J. Appl. Math. 67, 2002, 1–21. \\

\bibitem[ANS91]{ANS91} M. Akamatsu, G. Nakamura, S. Steinberg, Identification of Lam\'e coefficients from boundary observations
 \textit{Inverse Problems}, 7, 335--354, 1991. \\

\bibitem[Ar01]{Ar01}
T. Arens, Linear sampling methods for 2D inverse elastic wave scattering, 
Inverse Problems 17 (2001), 1445–1464. \\

\bibitem[AL98]{AL98} S. Arridge, W. Lionheart, Nonuniqueness in diffusion-based optical tomography,
Opt. Lett. 23, 1998. \\


\bibitem[BYZ19]{BYZ19}
G. Bao, T. Yin, F. Zeng,
Multifrequency Iterative Methods for the Inverse Scattering Medium Problems in Elasticity,
SIAM Journal on Scientific Computing, 41(4): B721–B745, 2019. \\

\bibitem[BHFVZ17]{BHFVZ17}
E. Beretta, M. V de Hoop, E. Francini, S. Vessella, J. Zhai,
Uniqueness and Lipschitz stability of an inverse boundary value problem for time-harmonic elasti
c waves,
Inverse Problems, 33, 035013, 2017. \\

\bibitem[BHQ13]{BHQ13}
E. Beretta,M. de Hoop M, L. Qiu, Lipschitz stability of an inverse boundary value problem for ti
me-harmonic elastic
waves, part I: recovery of the density,
Proceedings of the Project Review, Geo-Mathematical Imaging Group
(Purdue University, West Lafayette IN), 1, 263-272, 2013. \\

\bibitem[BC05]{BC05} M. Bonnet, A. Constantinescu, Inverse problems in elasticity,
Inverse Problems 21, 2005. \\

\bibitem[CC06]{CC06} F. Cakoni, D. Colton, Qualitative Methods in
Inverse Scattering Theory, Springer, Berlin, 2006. \\

\bibitem[CDGH20]{CDGH20}
V. Candiani, J. Darde, H. Garde, N. Hyvönen,
Monotonicity-Based Reconstruction of Extreme Inclusions in Electrical Impedance Tomography,
SIAM J. Math. Anal., 52, 6, 2020.\\

\bibitem[Ce96]{Ce96} M. Cessenat, Mathematical Methods in Electromagnetism,
World Scientific, Singapore, 1996. \\

\bibitem[Ci88]{Ci88} P. Ciarlet, Mathematical Elasticity, Volume: Three-dimensional Elasticity,
North-Holland, Amsterdam, 1988.\\

\bibitem[CK98]{CK98}
D. Colton, R. Kress, Inverse acoustic and electromagnetic scattering theory (Vol. 93, pp. xii+-334). Berlin, Springer, 1998. \\

\bibitem[EH21]{EH21} S. Eberle, B. Harrach, 
Shape reconstruction in linear elasticity: Standard and linearized monotonicity method,
Inverse Problems, 37(4):045006, 2021. \\

\bibitem[EHMR21]{EHMR21} S. Eberle, B. Harrach, H. Meftahi, and T. Rezgui, 
Lipschitz Stability Estimate and Reconstruction of Lamé Parameters in Linear  Elasticity, 
Inverse Probl. Sci. Eng., 29 (3), 396-417, 2021.\\

\bibitem[EM21]{EM21} S. Eberle, J. Moll, 
Experimental Detection and Shape Reconstruction of Inclusions in Elastic Bodies via a Monotonicity Method, 
Int. J. Solids Struct., 233, 111169, 2021.\\

\bibitem[EH22]{EH22} S. Eberle, B. Harrach, 
Monotonicity-Based Regularization for Shape Reconstruction in Linear Elasticity, 
Comput. Mech.,  69, 1069-1086, 2022\\

\bibitem[EH23]{EH23} S. Eberle-Blick, B. Harrach,
Resolution Guarantees for the Reconstruction of Inclusions in Linear Elasticity Based on Monotonicity Methods, 
Inverse Problems, 39, 075006, 2023.\\

\bibitem[EH23a]{EH23a} S. Eberle-Blick, N. Hyvönen, 
Bayesian Experimental Design for Linear Elasticity, 
Inverse Problems and Imaging, 18 (6), 1294-1319, 2024.\\

\bibitem[EP24]{EP24} S. Eberle-Blick, V. Pohjola,
The monotonicity method for inclusion detection and the time harmonic elastic wave equation,
Inverse Problems, 40, 045018, 2024.\\

\bibitem[EH19]{EH19}
J. Elschner, G. Hu,
Uniqueness and factorization method for inverse elastic scattering with a single incoming wave,
Inverse Problems, 35(9): 094002, 2019. \\


\bibitem[ER02]{ER02} G. Eskin,  J. Ralston,
On the inverse boundary value problem for linear isotropic elasticity,
Inverse Problems, 18, 907--922, 2002. \\

\bibitem[Fu20a]{Fu20a}T. Furuya, The factorization and monotonicity method for the defect in an
open periodic waveguide,
Journal of Inverse and Ill-posed Problems, vol. 28, no. 6, pp. 783-796, 2020. \\

\bibitem[Fu20b]{Fu20b}T. Furuya,  The monotonicity method for the inverse crack scattering probl
em, Inverse Problems in Science and Engineering 28(11):1-12, 2020.\\

\bibitem[Ge08]{Ge08}
B.\ Gebauer, Localized potentials in electrical impedance tomography,
\textit{Inverse Probl. Imaging} 2, 251--269, 2008. \\

\bibitem[GK08]{GK08}
N. Grinberg, A. Kirsch, The factorization method for inverse problems,
Oxford University Press, USA, 2008. \\

\bibitem[GH18]{GH18} R. Griesmaier, B. Harrach,
Monotonicity in inverse medium scattering on unbounded domains, SIAM J. Appl. Math. 78, 2533–255
7, 2018. \\

\bibitem[Ha09]{Ha09} B. Harrach, On uniqueness in diffuse optical tomography,
Inverse problems, 25, 2009. \\

\bibitem[HLL18]{HLL18}
B. Harrach, Y. Lin, H. Liu,
On Localizing and Concentrating Electromagnetic Fields,
SIAM J. Appl. Math., 78, 5, 10.1137/18M1173605, 2018. \\

\bibitem[HPS19a]{HPS19a} B. Harrach, V. Pohjola, M. Salo,
Dimension bounds in monotonicity methods for the Helmholtz equation,
\textit{SIAM J. Math. Anal.} 51-4, pp. 2995-3019, 2019. \\

\bibitem[HPS19b]{HPS19b}
B. Harrach, V. Pohjola, M. Salo,
Monotonicity and local uniqueness for the Helmholtz equation in a bounded domain,
\textit{Anal. PDE}, Vol. 12, No. 7, 1741--1771, 2019. \\

\bibitem[HS10]{HS10}
B. Harrach, J. K. Seo: Exact Shape-Reconstruction by One-Step Linearization in Electrical Impedance Tomography
SIAM J. Math. Anal. 42 (4), 1505–1518, 2010.\\ 

\bibitem[HU13]{HU13} B. Harrach, M. Ullrich,
Monotonicity-based shape reconstruction in electrical impedance tomography,
\emph{SIAM J. Math. Anal.} Vol. 45, No. 6, pp. 3382–3403, 2013.\\

\bibitem[HK15]{HK15}  F. Hettlich, A. Kirsch,
The Mathematical Theory of Time-Harmonic Maxwell's Equations,
Springer, Berlin, 2015. \\

\bibitem[HKS12]{HKS12}
G. Hu, A. Kirsch, M.  Sini,
Some inverse problems arising from elastic scattering by rigid obstacles.
Inverse Problems, 29(1): 015009, 2012. \\


\bibitem[HLZ13]{HLZ13}
G. Hu, Y. Lu, B. Zhang,
The factorization method for inverse elastic scattering from periodic structures, Inverse Proble
ms, 29(11): 115005, 2013. \\

\bibitem[IS23]{IS23} J. Ilmavirta, H. Schlüter, Gauge freedoms in the anisotropic elastic Dirichlet-to-Neumann map,
https://arxiv.org/abs/2309.15666, 2023. \\

\bibitem[Ik98]{Ik98} M. Ikehata,
Size estimation of inclusion, J. Inverse Ill-Posed Probl. 6:2, 127–140, 1998. \\

\bibitem[Ik90]{Ik90} M. Ikehata, Inversion formulas for the linearized problem for an inverse
boundary value problem in elastic prospection,
\textit{SIAM J. Appl. Math.} 50, 1635--1644, 1990. \\

\bibitem[Ik99]{Ik99}
M. Ikehata, How to draw a picture of an unknown inclusion from boundary measurements. Two mathematical
inversion algorithms, 
Journal of Inverse and Ill-Posed Problems, 7(3):255--271, 1999. \\

\bibitem[II08]{II08}
M. Ikehata, H. Itou,
Enclosure method and reconstruction of a linear crack in an elastic body,
Journal of Physics: Conference Series 135, 012052, 2008. \\

\bibitem[Kr64]{Kr64} M. A. Krasnosel'skij,  
Topological Methods in the Theory of Nonlinear Integral Equations. Oxford: Pergamon Press, 1964. \\

\bibitem[LL81]{LL81} L. Landau, E. Lifschitz, Theory of Elasticity, Pergamon, 1981. \\

\bibitem[Mc00]{Mc00} W. McLean,
Strongly Elliptic Systems and Boundary Integral Equations, Cambridge University Press, 2000. \\

\bibitem[Na88]{Na88}
A. Nachman, Reconstructions from boundary measurements Ann. Math.
128, 531–76, 1988.\\

\bibitem[NU94]{NU94} G. Nakamura, G. Uhlmann, Erratum: Global uniqueness for an inverse boundary
 value problem arising in elasticity,
\textit{Invent. Math.} 152,  205--207, 2003 (Erratum to Invent. Math. 118 (3) (1994) 457--474). \\

\bibitem[OSY92]{OSY92} O. Oleinik, S.Shamaev, G. Yosifian,
Mathematical problems inelasticity and homogenization,
North-Holland, Amsterdam, 1992. \\

\bibitem[Po22]{Po22} V. Pohjola,
On quantitative Runge approximation for the time harmonic Maxwell equations, 
\textit{Trans. Amer. Math. Soc.} 375 , 5727-5751, 2022. \\

\bibitem[Pot06]{Pot06}
R. Potthast, A survey on sampling and probe methods for inverse problems. Inverse Problems, 22(2), 2006. \\

\bibitem[Sc18]{Sc18} B. Schweizer, On Friedrichs Inequality, Helmholtz Decomposition, Vector Potentials, and the div-curl Lemma,
In: Rocca E., Stefanelli U., Truskinovsky L., Visintin A. (eds),
Trends in Applications of Mathematics to Mechanics, Springer INdAM Series, vol 27. Springer, 2018. \\

\bibitem[SFHC14]{SFHC14}
J. Shi, F. Faucher , M. de Hoop, H. Calandra,  Multi-level elastic full waveform inversion in isotropic media via
quantitative Lipschitz stability estimates. Proceedings of the Project Review, Geo-Mathematical Imaging Group,
Chicago, United States. 1:1-34, 2014. \\

\bibitem[TR02]{TR02}
A. Tamburrino and G. Rubinacci,
A new non-iterative inversion method for electrical resistance tomography,
\textit{Inverse Problems}, 6, 18,  1809--1829, 2002. \\

\bibitem[Ta06]{Ta06} A. Tamburrino, Monotonicity based imaging methods for elliptic and parabolic
inverse problems.
\emph{Journal of Inverse and Ill-posed Problems}, 14(6):633–642, 2006. \\

\bibitem[Zh10]{Zh10}
T. Zhou, Reconstructing electromagnetic obstacles by the enclosure method. 
Inverse Problems \& Imaging, 2010, 4 (3) : 547-569. \\


\end{thebibliography}
\end{document}